  \documentclass[11pt,letterpaper,oneside]{amsart}

  \usepackage{amsmath,amssymb,amsthm,amsfonts}
  \usepackage{enumerate}
  \usepackage[margin=1.2in]{geometry}
  \usepackage{xspace,xcolor}
  \usepackage[
     breaklinks,colorlinks,
     citecolor=blue,linkcolor=blue,urlcolor=teal
  ]{hyperref}
  
  \usepackage{tikz}
  \usepackage{tkz-euclide}

  \usetikzlibrary{arrows,calc,decorations.pathmorphing,backgrounds,positioning,fit}

  \usepackage{etoolbox}
  \makeatletter
  \patchcmd{\@secnumfont}{\mdseries}{\@empty}{}{}
  \patchcmd{\section}{\normalfont\scshape}{\large\bfseries}{}{}

  \makeatother

  \newcommand{\calA}{\mathcal{A}}
  
  \newcommand{\calC}{\mathcal{C}}

  \newcommand{\calG}{\mathcal{G}}

  \newcommand{\calL}{\mathcal{L}}
  
  \newcommand{\calN}{\mathcal{N}}
  
  \newcommand{\calP}{\mathcal{P}}
  \newcommand{\calQ}{\mathcal{Q}}
  
  \newcommand{\calS}{\mathcal{S}}
  \newcommand{\calT}{\mathcal{T}}

  \newcommand{\calX}{\mathcal{X}}
  \newcommand{\calY}{\mathcal{Y}}

  \newcommand{\HH}{\mathbb{H}}

  \newcommand{\RR}{\mathbb{R}}

  \newcommand{\ZZ}{\mathbb{Z}}

  \newtheorem{theorem}{Theorem}[section]
  \newtheorem{proposition}[theorem]{Proposition}
  \newtheorem{corollary}[theorem]{Corollary}
  \newtheorem{lemma}[theorem]{Lemma}

  \theoremstyle{definition}
  \newtheorem{definition}[theorem]{Definition}
  \newtheorem{remark}[theorem]{Remark}
  \newtheorem{claim}[theorem]{Claim}
  \newtheorem*{claim*}{Claim}
  
  \newtheorem{example}[theorem]{Example}
  
  \newtheorem*{question*}{Question}
  \newtheorem*{answer*}{Answer}
  \newtheorem*{application*}{Application}

  \newcommand{\secref}[1]{Section~\ref{Sec:#1}}
  \newcommand{\subsecref}[1]{Subsection~\ref{Subsec:#1}}
  \newcommand{\thmref}[1]{Theorem~\ref{Thm:#1}}
  
  \newcommand{\lemref}[1]{Lemma~\ref{Lem:#1}}
  \newcommand{\propref}[1]{Proposition~\ref{Prop:#1}}
  \newcommand{\clmref}[1]{Claim~\ref{Clm:#1}}
  
  \newcommand{\exaref}[1]{Example~\ref{Exa:#1}}

  \newcommand{\figref}[1]{Figure~\ref{Fig:#1}}
  \newcommand{\defref}[1]{Definition~\ref{Def:#1}}
  
  \newcommand{\eqnref}[1]{Equation~\ref{Eqn:#1}}

  \DeclareMathOperator{\diam}{diam}
  \DeclareMathOperator{\diamS}{diam_{\cc}}
  \DeclareMathOperator{\str}{stretch}
  \DeclareMathOperator{\I}{i}
  \DeclareMathOperator{\twist}{twist}

  \newcommand{\emul}{\stackrel{{}_\ast}{\asymp}}
  \newcommand{\gmul}{\stackrel{{}_\ast}{\succ}}
  \newcommand{\lmul}{\stackrel{{}_\ast}{\prec}}
  \newcommand{\eadd}{\stackrel{{}_+}{\asymp}}
  \newcommand{\gadd}{\stackrel{{}_+}{\succ}}
  \newcommand{\ladd}{\stackrel{{}_+}{\prec}}

  \newcommand{\s}{\ensuremath{\calS}\xspace} 
  \newcommand{\torus}{\ensuremath{\s_{1,1}}\xspace} 
  \newcommand{\T}{\ensuremath{\calT(\s)}\xspace} 
   
  \newcommand{\cc}{\ensuremath{\calC(\s)}\xspace} 

  \newcommand{\Teich}{{Teichm\"uller }} 

  \newcommand{\param}{{\mathchoice{\mkern1mu\mbox{\raise2.2pt\hbox{$
  \centerdot$}}
  \mkern1mu}{\mkern1mu\mbox{\raise2.2pt\hbox{$\centerdot$}}\mkern1mu}{
  \mkern1.5mu\centerdot\mkern1.5mu}{\mkern1.5mu\centerdot\mkern1.5mu}}}

  \renewcommand{\setminus}{{\smallsetminus}}
  
  \newcommand{\st}{\mathbin{\mid}} 
  \newcommand{\ST}{\mathbin{\Big|}} 
  \newcommand{\from}{\colon\thinspace} 
  \newcommand{\ep}{\ensuremath{\epsilon}\xspace}

  \newcommand{\GL}{\ensuremath{\calG}\xspace} 
  \newcommand{\dL}{\ensuremath{d_{\text{Th}}}\xspace} 
  \newcommand{\dH}{\ensuremath{d_{\HH}}\xspace} 
  \newcommand{\Lip}{\ensuremath{\text{L}}\xspace} 
   
  \newcommand{\dS}{\ensuremath{d_{\cc}}\xspace}

  \newcommand{\balpha}{{\overline \alpha}}

  \newcommand{\bomega}{{\overline \omega}}
  \newcommand{\blambda}{{\overline \lambda}}
  \newcommand{\etab}{{\overline \eta}}
  
  \newcommand{\talpha}{{\widetilde\alpha}}
  \newcommand{\tbeta}{{\widetilde\beta}}
  \newcommand{\tlambda}{{\widetilde\lambda}}
  \newcommand{\tomega}{{\widetilde\omega}}
  \newcommand{\tgamma}{{\widetilde\gamma}}
  \newcommand{\teta}{{\widetilde\eta}}
  \newcommand{\tf}{{\widetilde f}}

  \newcommand{\tx}{{\widetilde X}}
  \newcommand{\ty}{{\widetilde Y}}
  
  \newcommand{\hgline}[2]{
  \pgfmathsetmacro{\thetaone}{#1}
  \pgfmathsetmacro{\thetatwo}{#2}
  \pgfmathsetmacro{\theta}{(\thetaone+\thetatwo)/2}
  \pgfmathsetmacro{\phi}{abs(\thetaone-\thetatwo)/2}
  \pgfmathsetmacro{\close}{less(abs(\phi-90),0.0001)}
  \ifdim \close pt = 1pt
      \draw (\theta+90:1) -- (\theta-90:1);
  \else
      \pgfmathsetmacro{\R}{tan(\phi)}
      \pgfmathsetmacro{\distance}{sqrt(1+\R^2)}
      \draw (\theta:\distance) circle (\R);
  \fi
  }
  
\begin{document}


  \title    {The shadow of a Thurston geodesic to the curve graph } \author
  {Anna Lenzhen} \author   {Kasra Rafi} \author   {Jing Tao}
  
  \thanks{The authors acknowledge support from U.S. National Science
    Foundation grants DMS \#1107452, \#1107263, \#1107367 ``RNMS: Geometric
    structures And Representation varieties" (the GEAR Network). The second
    author is partially supported by NCERC \#435885. The third author is
    partially supported by NSF DMS \#1311834}
  

  \begin{abstract} 
    
    We study the geometry of the Thurston metric on \Teich space by
    examining its geodesics and comparing them to \Teich geodesics. We
    show that, similar to a \Teich geodesic, the shadow of a Thurston geodesic to
    the curve graph is a reparametrized quasi-geodesic. However, we show
    that the set of short curves along the two geodesics
    are not identical. 

  \end{abstract}
  
  \maketitle
  

\section{Introduction}
  
  \label{Sec:Intro}

  In \cite{Thu86a} Thurston introduced a metric on \Teich space in terms of
  the least possible value of the global Lipschitz constant between two
  hyperbolic surfaces of finite volume. Even though this is an asymmetric
  metric, Thurston constructed geodesics connecting any pair of points in
  \Teich space that are concatenations of \emph{stretch paths}. However,
  there is no unique geodesic connecting two points in \Teich space $\T$.
  We construct some examples to highlight the extent of non-uniqueness of
  geodesics:
   
  \begin{theorem} \label{Thm:Bad-Geodesic}

    For every $D>0$, there are points $X,Y,Z \in \T$ and Thurston geodesic
    segments $\GL_1$ and $\GL_2$ starting from $X$ and ending in $Y$ with
    the following properties:
    \begin{enumerate}
      
      \item Geodesics $\GL_1$ and $\GL_2$ do not fellow travel each other;
      the point $Z$ lies in path $\GL_1$ but is at least $D$ away from any
      point in $\GL_2$. 
      
      \item The geodesic $\GL_1$ parametrized in any way in the reverse
      direction is not a geodesic. In fact, the point $Z$ is at least $D$
      away from any point in any geodesic connecting $Y$ to $X$. 

    \end{enumerate}

  \end{theorem}
  
  In view of these examples, one may ask whether geodesics connecting $X$
  to $Y$ have any common features. There is a mantra that all notions of a
  straight line in \Teich space behave the same way at the level of the
  curve graph. That is, the shadow of any such line to the curve graph is a
  \emph{reparametrized quasi-geodesic}. This has already been shown for
  \Teich geodesics \cite{MM99}, lines of minima \cite{CRS08}, grafting rays
  \cite{CDR12}, certain geodesics in the Weil-Petersson metric
  \cite{BMM11}, and Kleinian surface groups \cite{Min10}. (See also
  \cite{BF11} of an analogous result in Outer Space.) 
  
  In this paper, we show
  
  \begin{theorem} \label{Thm:Shadow}

    The shadow of a Thurston geodesic to the curve graph is a
    reparametrized quasi-geodesic. 

  \end{theorem}
  
  Since the curve graph is Gromov hyperbolic \cite{MM99}, quasi-geodesics with 
  common endpoints fellow travel. Hence:
  
  \begin{corollary} \label{Cor:Shadow}

    The shadow to the curve graph of different Thurston geodesics connecting 
    $X$ to $Y$  fellow travel each other. 

  \end{corollary}
  
  This builds on the analogy established in \cite{LRT12} between \Teich
  geodesics and Thurston geodesics. We showed that if the \Teich geodesic
  connecting $X$ and $Y$ stays in the thick part of \Teich space, so does
  any Thurston geodesic connecting $X$ to $Y$ and in fact all these paths
  fellow travel each other. However, this analogy does not extend much
  further; we show that the converse of the above statement is not true:

  \begin{theorem} \label{Thm:Not-Short}

    There is an $\ep_0>0$ such that, for every $\ep>0$, there are points
    $X,Y \in \T$ and a Thurston geodesic connecting $X$ to $Y$ that stays
    in the $\ep_0$--thick part of \Teich space whereas the associated
    \Teich geodesic connecting $X$ to $Y$ does not stay in the $\ep$--thick
    part of \Teich space. 

  \end{theorem}  

  In particular, this means that the set of short curves along a \Teich
  geodesic and a Thurston geodesic are not the same. 
  
  \subsection*{Outline of the proof} 
  
  To prove \thmref{Shadow} one needs a suitable definition for when a curve
  is \emph{sufficiently horizontal} along a Thurston geodesic. This is in
  analogy with both the study of \Teich geodesics and geodesics in Outer
  space after \cite{MM99, BF11}. In the \Teich metric, geodesics are
  described by a quadratic differential, which in turn defines a singular
  flat structure on a Riemann surface. The flat metric is then deformed by
  stretching the horizontal foliation and contracting the vertical
  foliation of the flat surface. If a curve is not completely vertical,
  then its horizontal length grows exponentially fast along the \Teich
  geodesic. Similarly, for Outer space, the geodesics are described as
  folding paths associated to train-tracks. If an immersed curve has a
  sufficiently long \emph{legal segment} at a point along a folding path,
  then the length of the horizontal segment grows exponentially fast along
  the folding path. These two facts respectively play important roles in
  the proofs of the hyperbolicity of curve complexes and free factor
  complexes.

  The notion of horizontal foliation in the setting of \Teich geodesics is
  replaced by the \emph{maximally stretched lamination} in the setting of
  Thurston geodesics (see \subsecref{metric}). However, it is possible for
  a curve $\alpha$ on the surface to fellow travel the maximally stretched
  lamination $\lambda$ for a long time only to have its length go down
  later along the Thurston geodesic. That is, the property of fellow
  traveling $\lambda$ geometrically does not persist (see \exaref{Shear}).

  In \secref{Horizontal}, we define the notion of a curve $\alpha$ being
  \emph{sufficiently horizontal} along a Thurston geodesic to mean that
  $\alpha$ fellow travels $\lambda$ for sufficiently long time both
  topologically and geometrically. We show that if a curve is sufficiently
  horizontal at a point along a Thurston geodesic, then it remains
  sufficiently horizontal throughout the geodesic, with exponential growth
  of the length of its horizontal segment (\thmref{Horizontal}). In
  \secref{Shadow}, we define a projection map from the curve complex to the
  shadow of a Thurston geodesic sending a curve $\alpha$ first to the
  earliest time in the Thurston geodesic where $\alpha$ is sufficiently
  horizontal and then to a curve of bounded length at that point. We show
  in \thmref{Retraction} that this map is a coarse Lipschitz retraction.
  \thmref{Shadow} follows from this fact using a standard argument. In
  \secref{Examples}, we construct the examples of Thurston geodesics that
  illustrate the deviant behaviors of Thurston geodesics from \Teich
  geodesics, as indicated by \thmref{Bad-Geodesic} and \thmref{Not-Short}.
 
  The proof is somewhat technical, because all we know about a Thurston
  geodesic is that the length of the maximally stretched lamination (which
  may not be a filling lamination) is growing exponentially. Using this and
  some delicate hyperbolic geometry arguments, we are able to control the
  geometry of the surface. For the ease of exposition, we have collected
  several technical lemmas in \secref{Hyperbolic}. These statements should
  be intuitively clear to a reader familiar with hyperbolic geometry and
  the proofs can be skipped in the first reading of the paper.
  \secref{Examples} is also less technical and can be read independently
  from the rest of the paper.
 
  \subsection*{Acknowledgement} 

  Our Key \propref{Balanced} is modeled after \cite[Proposition 6.4]{BF11}
  where Bestvina-Feighn show the projection of a folding path to the free
  factor graph is a reparametrized quasi-geodesic which is in turn inspired
  by the arguments of Masur-Minsky \cite{MM99}. We would like to thank
  Universit\'e de Rennes and Erwin Schr\"odinger International Institute
  for Mathematical Physics for their hospitality. We also thank the referee
  for carefully reading the paper and providing us all the useful comments. 
  
\section{Background}
  
  \label{Sec:Background}

  We briefly review some background material needed for this paper. We
  refer to \cite{Thu86a, Pap07, Hub06} and the references therein for
  background on hyperbolic surfaces and the Thurston metric on \Teich
  space. 
  
  \subsection{Notation}
  
  We adopt the following notation to simplify some calculations. Call a
  constant $C$ \emph{universal} if it depends only on the topological type
  of a surface, and not on a hyperbolic metric on the surface. Then given a
  universal constant $C$ and two quantities $a$ and $b$, we write
  \begin{itemize} \item $a \lmul b$ if $a \le Cb$. \item $a \emul b$ if $a
  \lmul b$ and $b \lmul a$. \item $a \ladd b$ if $a \le b +C$. \item $a
  \eadd b$ if $a \ladd b$ and $b \ladd a$. \item $a \prec b$ if $a \le
  Cb+C$. \item $a \asymp b$ if $a \prec b$ and $b \prec a$. \end{itemize}
  We will also write $a = O(1)$ to mean $a \lmul 1$.   

  \subsection{Coarse maps}

  Given two metric spaces $\calX$ and $\calY$, a multivalued map $f \colon
  \calX \to \calY$ is called a \emph{coarse map} if the image of every
  point has uniformly bounded diameter. The map $f$ is \emph{(coarsely)
  Lipschitz} if $d_\calY \big( f(x),f(y) \big) \prec d_\calX(x,y)$ for all
  $x,y \in \calX$, where \[ d_\calY \big( f(x), f(y) \big) = \diam_\calY
  \big( f(x) \cup f(y) \big).\] Given a subset $\calA \subset \calX$, a
  coarse Lipschitz map $f : \calX \to \calA$ is a \emph{coarse retraction}
  if $d_\calX \big(a, f (a) \big) =O(1)$ for all $a \in \calA$.   
  
  \subsection{Curve graph}

  Let \s be a connected oriented surface of genus $g$ with $p$ punctures
  with $3g+p-4 \ge 0$. By a curve on \s we will mean an essential simple
  closed curve up to free homotopy. \emph{Essential} means the curve is not
  homotopic to a point or a puncture of \s. For two curves $\alpha$ and
  $\beta$, let $\I(\alpha,\beta)$ be the minimal intersection number
  between the representatives of $\alpha$ and $\beta$. Two distinct curves
  are disjoint if their intersection number is 0. A \emph{multicurve} on \s
  is a collection of pairwise disjoint curves. A \emph{pair of pants} is
  homeomorphic to a thrice-punctured sphere. A \emph{pants decomposition}
  on \s is a multicurve whose complement in \s is a disjoint union of pairs
  of pants.
  
  We define the \emph{curve graph} \cc of \s as introduced by Harvey
  \cite{Har81}. The vertices of $\cc$ are curves on \s, and two curves span
  an edge if they intersect minimally on \s. For a surface with $3g+p-4 >
  0$, the minimal intersection number is 0; for the once-punctured torus,
  the minimal intersection number is 1; and for the four-times punctured
  sphere, the minimal intersection number is 2. The curve graph of a pair
  of pants is empty since there are no essential curves. By an element or a
  subset of \cc, we will always mean a vertex or a subset of the vertices
  of \cc.   

  Assigning each edge of \cc to have length 1 endows \cc with a metric
  structure. Let $\dS(\param,\param)$ be the induced path metric on $\cc$.
  The following fact will be useful for bounding curve graph distances
  \cite{Sch06}: for any $\alpha, \beta \in \cc$, \begin{equation}
  \label{Eqn:Intersection} \dS(\alpha,\beta) \le \log_2 \I(\alpha,
  \beta)+1. \end{equation}
  
  By \cite{MM99}, for any surface \s, the graph $\cc$ is hyperbolic in the
  sense of Gromov. More recently, it was shown contemporaneously and
  independently by \cite{Aou13, Bow13, HPW13, CRS13} that there is a
  uniform $\delta$ such that $\cc$ is $\delta$--hyperbolic for all \s. 
 
  \subsection{\Teich space}
  
  A \emph{marked} hyperbolic surface is a complete finite-area hyperbolic
  surface equipped with a fixed homeomorphism from \s. Two marked
  hyperbolic surfaces $X$ and $Y$ are considered equivalent if there is an
  isometry from $X$ to $Y$ in the correct homotopy class. The collection of
  equivalence classes of all marked hyperbolic surfaces is called the
  \Teich space \T of \s. This space \T equipped with its natural topology
  is homeomorphic to $\RR^{6g-6+2p}$.
 
  \subsection{Short curves and collars} \label{Subsec:collar}

  Given $X \in \T$ and a simple geodesic $\omega$ on $X$, let
  $\ell_X(\omega)$ be the arc length of $\omega$. Since $X$ is marked by a
  homeomorphism to \s, its set of curves is identified with the set of
  curves on \s. For a curve $\alpha$ on \s, let $\ell_X(\alpha) =
  \ell_X(\alpha^*)$, where $\alpha^*$ is the geodesic representative of
  $\alpha$ on $X$. A curve is called a \emph{systole} of $X$ if its
  hyperbolic length is minimal among all curves. Given a constant $C$, a
  multicurve on $X$ is called $C$--short if the length of every curve in
  the set is bounded above by $C$. The \emph{Bers constant}
  $\ep_B=\ep_B(\s)$ is the smallest constant such that every hyperbolic
  surface $X$ admits an $\ep_B$--short pants decomposition. In most
  situations, we will assume a curve or multicurve is realized by geodesics
  on $X$. 
  
  We state the well-known Collar Lemma with some additional properties (see
  \cite[\S 3.8]{Hub06}).
  
  \begin{lemma}[Collar Lemma] \label{Lem:Collar} 

    Let $X$ be a hyperbolic surface. For any simple closed geodesic
    $\alpha$ on $X$, the regular neighborhood about $\alpha$    \[
    U(\alpha) = \left\{ p \in X \st d_X(p,\alpha) \le  \sinh^{-1}
    \frac{1}{\sinh \big( .5 \ell_X(\alpha) \big)} \right\} \] is an
    embedded annulus. If two simple closed geodesics $\alpha$ and $\beta$
    are disjoint, then $U(\alpha)$ and $U(\beta)$ are disjoint. Moreover,
    given a simple closed geodesic $\alpha$ and  a simple geodesic $\omega$
    (not necessarily closed, but complete), if $\omega$ does not intersect
    $\alpha$ and does not spiral towards $\alpha$, then it is disjoint from
    $U(\alpha)$.

  \end{lemma}
 
  We will refer to $U(\alpha)$ as \emph{the standard collar of $\alpha$.}
  There is a universal upper and lower bound on the arc length of the
  boundary of $U(\alpha)$ provided that $\alpha$ is $\ep_B$--short. 

  Using the convention \begin{equation}\label{Eqn:Log} \log(x) =
  \begin{cases} \ln(x) & \text{if $x \ge e$} \\ 1 & \text{if $x\leq e$},
  \end{cases} \end{equation} we note that, for $0 \le x \le \ep_B$,   \[
  \sinh^{-1} \big(1/ \sinh(.5 x) \big) \eadd \log(1/x). \] 
  
  A consequence of the Collar Lemma is the existence of a universal
  constant $\delta_B$ such that, if a curve $\beta$ intersects an
  $\ep_B$--short curve $\alpha$, then \[ \ell_X(\beta)\geq
  \I(\alpha,\beta) \delta_B. \] Also for any geodesic segment $\omega$, \[
  \ell_X(\omega) \ge (\I(\alpha, \omega)-1) \delta_B. \] We will refer to
  $\delta_B$ as the \textit{dual constant} to the Bers constant $\ep_B$. 

  \subsection{Various notions of twisting}
  
  In this section, we will define several notions of \emph{relative
  twisting} of two objects or structures about a simple closed curve
  $\gamma$. The notation will always be $\twist_\gamma(\param,\param)$. 
  
  First suppose $A$ is a compact annulus and $\gamma$ is the core curve of
  $A$. Given two simple arcs $\eta$ and $\omega$ with endpoints on the
  boundary of $A$, we define \[ \twist_\gamma(\eta,\omega) =
  \I(\eta,\omega), \] where $\I(\eta,\omega)$ is the minimal number of
  interior intersections between isotopy classes of $\eta$ and $\omega$
  fixing the endpoints pointwise.   

  Now suppose $\gamma$ is a curve in \s. The annular cover $\hat A$ of \s
  corresponding to $\langle \gamma \rangle < \pi_1(\s)$ can be compactified
  in an intrinsic way. Let $\hat \gamma$ be the core curve of $\hat A$.
  Given two simple geodesics or curves $\eta$ and $\omega$ in \s, let
  $\hat{\eta}$ and $\hat{\omega}$ be any lifts to $\hat{A}$ that join the
  boundary of $\hat A$ (such lifts exist when $\eta$ and $\omega$ intersect
  $\gamma$). The \emph{relative twisting} of $\eta$ and $\omega$ about
  $\gamma$ is \[ \twist_\gamma(\eta,\omega) = \twist_{\hat \gamma}(\hat
  \eta, \hat \omega). \] This definition is well defined up to an additive
  error of $1$ with different choices of $\hat \eta$ and $\hat \omega$.
 
  Now suppose $X \in \T$ and let $\omega$ be a geodesic arc or curve in
  $X$. We want to measure the number of times $\omega$ twists about
  $\gamma$ in $X$. To do this, represent $\gamma$ by a geodesic and lift
  the hyperbolic metric of $X$ to the annular cover $\hat A$. Let $\hat
  \tau$ be any geodesic perpendicular to $\hat \gamma$ joining the boundary
  of $\hat A$. We define the twist of $\omega$ about $\gamma$ on $X$ to be
  \[ \twist_\gamma(\omega, X) = \I(\hat \omega, \hat \tau), \] where $\hat
  \omega$ is any lift of $\omega$ joining the boundary of $\hat A$. Since
  there may be other choices of $\hat\tau$, this notion is well defined up
  to an additive error of at most one. Note that if
  $\twist_\gamma(\omega,X) = 0$ and $\twist_\gamma (\eta,\omega) = n$, then
  $\twist_\gamma(\eta,X) \eadd n$.
  
  When $\gamma$ is $\ep_B$--short, fix a perpendicular arc $\tau$ to the
  standard collar $U(\gamma)$, then the quantity $\I(\omega, \tau)$ differs
  from $\twist_\gamma(\omega,X)$ by at most one \cite[Lemma 3.1]{Min96a}. 
  
  Given $X, Y \in \T$, the relative twisting of $X$ and $Y$ about $\gamma$
  is \[ \twist_\gamma(X,Y) = \I(\hat \tau_X, \hat \tau_Y),\] where $\hat
  \tau_X$ is an arc perpendicular to $\hat \gamma$ in the metric $X$, and
  $\hat \tau_Y$ is an arc perpendicular to $\hat \gamma$ in the metric $Y$.
  Again, choosing different perpendicular arcs changes this quantity by at
  most one.
  
  \subsection{Subsurface projection and bounded combinatorics}

  Let $\Sigma \subset \s$ be a compact and connected subsurface such that
  each boundary component of $\Sigma$ is an essential simple closed curve.
  We assume $\Sigma$ is not a pair of pants or an annulus. From
  \cite{MM00}, we recall the definition of subsurface projection
  $\pi_\Sigma \from \cc \to \calP \big( \calC(\Sigma) \big)$ from the curve
  graph of $S$ to the space of subsets of the curve graph of $\Sigma$. 
 
  Equip \s with a hyperbolic metric and represent $\Sigma$ as a convex set
  with geodesic boundary. (The projection map does not depend on the choice
  of the hyperbolic metric.) Let $\hat \Sigma$ be the Gromov
  compactification of the cover of \s corresponding to $\pi_1(\Sigma) <
  \pi_1(\s)$. There is a natural homeomorphism from $\hat \Sigma$ to
  $\Sigma$, which allows us to identify $\calC(\hat \Sigma)$ with
  $\calC(\Sigma)$. For any curve $\alpha$ on \s, let $\hat \alpha$ be the
  closure of the lift of $\alpha$ in $\hat \Sigma$. For each component
  $\beta$ of $\hat \alpha$, let $\calN_\beta$ be a regular neighborhood of
  $\beta \cup \partial \hat \Sigma$. The isotopy class of each component
  $\beta'$ of $\partial \calN_\beta$, with isotopy relative to $\partial
  \hat \Sigma$, can be regarded as an element of $\calP \big( \calC(\Sigma)
  \big)$; $\beta'$ is the empty set if $\beta'$ is isotopic into $ \partial
  \hat \Sigma$. We define \[ \pi_\Sigma(\alpha) = \bigcup_{\beta \subset
  \hat \alpha} \bigcup_{\beta' \subset \partial \calN_\beta} \{ \beta' \}.
  \] 

  The projection distance between two elements $\alpha, \beta \in \cc$ in
  $\Sigma$ is \[ d_{\calC(\Sigma)}(\alpha, \beta) = \diam_{\calC(\Sigma)}
  \big( \pi_\Sigma(\alpha) \cup \pi_\Sigma(\beta) \big). \] Given a subset
  $K \subset \cc$, we also define $\pi_\Sigma(K) = \bigcup_{\alpha \in K}
  \pi_\Sigma(\alpha)$, and the projection distance between two subsets of
  \cc in $\Sigma$ are likewise defined. For any $\Sigma \subset \s$, the
  projection map $\pi_\Sigma$ is a coarse Lipschitz map \cite{MM00}. 

  For any $X \in \T$, a pants decomposition $\calP$ on $X$ is called
  \emph{short} if $\sum_{\alpha \subset \calP}\ell_X(\alpha)$ is minimized.
  Note that a short pants decomposition is always $\ep_B$--short, and two
  short pants decompositions have bounded diameter in \cc. 
  
  Let $X_1, X_2 \in \T$. For $i=1,2$, let $\calP_i$ be a short pants
  decompositions on $X_i$. We will say $X_1$ and $X_2$ have
  \emph{$K$--bounded combinatorics} if there exists a constant $K$ such
  that the following two properties hold. \begin{itemize}

    \item For $\Sigma = \s$, or $\Sigma$ a subsurface of \s, \[
    d_{\calC(\Sigma)}(\calP_1,\calP_2)\leq K.\]

    \item For every curve $\gamma$ in \s \[ \twist_\gamma(X_1,X_2) \le K.
    \]
    
  \end{itemize}

  \subsection{Geodesic lamination}

  Let $X$ be a hyperbolic metric on \s. A geodesic lamination $\mu$ is a
  closed subset of  $\s$ which is a union of disjoint simple complete
  geodesics in the metric of $X$. These geodesics are called \emph{leaves}
  of $\mu$, and we will call their union  the \emph{support} of $\mu$. A
  basic example of a geodesic lamination is a multicurve (realized by its
  geodesic representative). 

  Given another hyperbolic metric on $\s$, there is a canonical one-to-one
  correspondence between the two spaces of geodesic laminations. We
  therefore will denote the space of geodesic laminations on  $\s$ by
  $\calG\calL(\s)$ without referencing to a hyperbolic metric. The set
  $\calG\calL(\s)$ endowed with the Hausdorff distance is compact. A
  geodesic lamination is said to be \emph{chain-recurrent} if it is in the
  closure of the set of all multicurves.

  A transverse measure on a geodesic lamination $\mu$ is a Radon measure on
  arcs transverse to the leaves of the lamination. The measure is required
  to be invariant under projections along the leaves of $\mu$. When $\mu$
  is a simple closed geodesic, the transverse measure is just the counting
  measure times a positive real number. It is easy to see that an infinite
  isolated leaf spiraling towards a closed leaf cannot be in the support of
  a transverse measure.
  
  The \emph{stump} of a geodesic lamination $\mu$ is a maximal (with
  respect to inclusion) compactly-supported sub-lamination of $\mu$ which
  admits a transverse measure of full support.  

  \subsection{Thurston metric} \label{Subsec:metric}

  In this section, we will give a brief overview of the Thurston metric,
  sometimes referred to as the Lipschitz metric or Thurston's
  ``asymmetric'' metric in the literature. All facts in this section are
  due to Thurston and contained in \cite{Thu86a}. We also refer to
  \cite{Pap07} for additional reference.
  
  Given $X, Y \in \T$, the distance \emph{from} $X$ to $Y$ in the Thurston
  metric is defined to be \[ \dL(X,Y) = \log \Lip(X,Y), \] where
  $\Lip(X,Y)$ is the infimum of Lipschitz constants over all homeomorphisms
  from $X$ to $Y$ in the correct homotopy class. Since the inverse of a
  Lipschitz map is not necessarily Lipschitz, there is no reason for the
  metric to be symmetric. In fact, $\Lip(X,Y)$ is in general not equal to
  $\Lip(Y,X)$, as shown in the example on page 5 of \cite{Thu86a}.  
 
  Thurston showed that the quantity $\Lip(X,Y)$ can be computed using
  ratios of lengths of curves on $\s$.
  
  \begin{theorem}[\cite{Thu86a}] \label{Thm:LengthRatio}
     For any $X, Y \in \T$, \[ L(X,Y) = \sup_{\alpha}
     \frac{\ell_Y(\alpha)}{\ell_X(\alpha)},\] where $\alpha$ ranges over
     all curves on \s. 
   \end{theorem}

  The length function extends continuously to measured laminations, and the
  space of projectivized measured laminations is compact. Hence there is a
  measured lamination that realizes the supremum above. It might not be
  unique, but one can assign to an ordered pair $(X,Y)$ a geodesic
  lamination $\mu(X,Y)$ admitting a transverse measure that contains the
  supports of all the measured laminations realizing the supremum. 
  
  For any sequence $\{ \alpha_i \}$ of curves on \s with $\displaystyle
  \lim_{i \to \infty} \ell_Y(\alpha_i)/ \ell_X(\alpha_i) \to L(X,Y)$, let
  $\alpha_\infty$ be a limiting geodesic lamination of $\{ \alpha_i \}$ in
  the Hausdorff topology. Set 
  \[ \lambda(X,Y) = \bigcup \{ \alpha_\infty \}, \] 
  where the union on the right-hand side ranges over the limits of all such
  sequences. Thurston showed that $\lambda(X,Y)$ is a geodesic lamination,
  called the \emph{maximally stretched lamination} from $X$ to $Y$, which
  contains $\mu(X,Y)$ as its stump. Moreover, there is a
  $\Lip(X,Y)$--Lipschitz homeomorphism from $X$ to $Y$ in the correct
  homotopy class that stretches $\lambda(X,Y)$ by $\Lip(X,Y)$ and whose
  local Lipschitz constant outside $\lambda(X,Y)$ is strictly less than
  $\Lip(X,Y)$. In particular, the infimum is realized in the definition of
  $L(X,Y)$ and $\dL(X,Y)$. The existence of such a map follows from the
  fact that one can connect $X$ to $Y$ by a concatenation of finitely many
  stretch paths (see \subsecref{Shear}), all of which contain
  $\lambda(X,Y)$ in its stretch locus. We will call a homeomorphism  $f
  \colon X \to Y$ \emph{optimal} if $f$ is a $\Lip(X,Y)$--Lipschitz map.
  Note that our sense of ``optimal'' is more in the sense of $L^\infty$
  metric than $L^1$, since we only require the global Lipschitz constant to
  be minimized.
  
  Thurston showed that \T equipped with the Thurston metric is a
  (asymmetric) geodesic metric space. That is, for any $X, Y \in \T$, there
  exists a geodesic from $X$ to $Y$, i.e.\ a parametrized path $\GL \colon
  [0, d] \to \T$ such that $d=\dL(X,Y)$, $\GL(0) = X$, $\GL(d) = Y$, and
  for any $0 \le s < t \le d$, $\dL(\GL(s),\GL(t)) = t-s$. Any geodesic
  from $X$ to $Y$ is characterized by the property that the maximally
  stretched lamination $\lambda(X,Y)$ is stretched maximally at all times.
  Thus, there is only one such geodesic only when $\lambda(X,Y)$ is a
  maximal lamination (the complement of $\lambda(X,Y)$ are ideal
  triangles). In general, the set of geodesics from $X$ to $Y$ can have
  uncountable cardinality: the idea is that one is free to deform any part
  of the surface which is not forced to be maximally stretched. We refer to
  the proof of \thmref{Bad-Geodesic} in \secref{Examples} for an example of
  such deformation. 
  
  Given a geodesic segment $\GL \from [a,b] \to \T$, we will often denote
  by $\lambda_{\GL}$ the maximally stretched lamination from $\GL(a)$ to
  $\GL(b)$. The maximally stretched lamination is well defined for geodesic
  rays or bi-infinite geodesics. Suppose $\GL \from \RR \to \T$ is a
  bi-infinite geodesic. Consider two sequences $\{t_n\}_n, \{s_m\}_m
  \subset \RR$ with 
  \[ \displaystyle \lim_{n \to \infty} t_n \to \infty \quad \text{and}
  \quad \lim_{m\to \infty} s_m \to -\infty.\] Set $X_m = \GL(s_m)$
  and $Y_n =\GL(t_n)$. The sequence $\lambda(X_m,Y_n)$ is increasing by
  inclusion as $m, n \to \infty$, hence $\lambda = \bigcup_{m,n}
  \lambda(X_m,Y_n)$ is defined. The lamination $\lambda$ is independent of
  the sequences $X_m$ and $Y_n$, hence $\lambda=\lambda_{\GL}$ is the
  maximally stretched lamination for $\GL$. Similarly, the maximally
  stretched lamination $\lambda_{\GL}$ is defined for a geodesic ray $\GL
  \from [a,\infty) \to \T$.
  
  Throughout this paper, we will always assume Thurston geodesics are
  parameterized by arc length.
  
  \subsection{Stretch Paths} \label{Subsec:Shear}

  To prove $\T$ is a geodesic metric space, Thurston introduced a special
  family of geodesics called \textit{stretch paths}. Namely, let $\lambda$
  be a maximal geodesic lamination (all complementary components are ideal
  triangles). Then $\lambda$, together with a choice of basis for relative
  homology, defines \emph{shearing coordinates} on \Teich space (see
  \cite{Bon97}). In fact, in this situation, the choice of basis does not
  matter. For any hyperbolic surface $X$, and time $t$, define $\str(X,
  \lambda, t)$ to be the hyperbolic surface where the shearing coordinates
  are $e^t$ times the shearing coordinates at $X$. That is, the path $t
  \mapsto \str(X, \lambda,t)$ is a straight line in the shearing coordinate
  system associated to $\lambda$. Thurston showed  \cite{Thu86a} that this
  path is a geodesic in $\T$. 
  
  In the case where $\lambda$ is a finite union of geodesics, the shearing
  coordinates are easy to understand. An ideal triangle has an inscribed
  circle tangent to each edge at a point which we refer to as an
  \emph{anchor point}. Then the shearing coordinate associated to two
  adjacent ideal triangles is the distance between the anchor points coming
  from the two triangles (see \figref{Shear}). To obtain the surface
  $\str(X, \lambda,t)$, one has to \emph{slide} every pair of adjacent
  triangles against each other such that the distance between the
  associated pairs of anchor points is increased by a factor of $e^t$. 

  \subsection{An example} \label{Subsec:Example}
 
  We illustrate some possible behaviors along a stretch path in the
  following basic example. 
 
  Fix a small $0<\epsilon\ll 1.$ Let $A_0$ be an annulus which is glued out
  of two ideal triangles as follows. One pair of the sides is glued with a
  shift of $2$, and another pair is glued with a shift of $2+2\epsilon$, as
  in \figref{Shear}. That is, if $p,p'$ are the anchor points associated to
  one triangle and $q,q'$ are anchor points associated to the other
  triangle, and the sides containing $p$ and $q$  are glued, and same for
  $p'$ and $q'$, then the segments $[p,q]$ and $[p',q']$ have lengths $2$
  and $2+2\epsilon$ respectively. Note that by adding enough ideal
  triangles to this construction and gluing them appropriately, one can
  obtain a hyperbolic surface  of arbitrarily large complexity. For example
  adding an ideal triangle to the top and to the bottom of $A_0$ with zero
  shift and then identifying the top edges with zero shift and the bottom
  edges with zero shift gives rise to a hyperbolic surface $X_0$ which is
  topologically a sphere with four points removed. Also, the sides of the
  four ideal triangles define a maximal geodesic lamination $\lambda$ on
  $X_0$. When two triangles are glued with zero shift, the associated
  shearing coordinate remains unchanged under a stretch map. Hence, we
  concentrate on how the geometry of $A_0$ changes only. 

  \begin{figure}[htp!] 
  \setlength{\unitlength}{0.01\linewidth}
  \begin{picture}(100, 37)
  \put(9,-2){ 
  \includegraphics{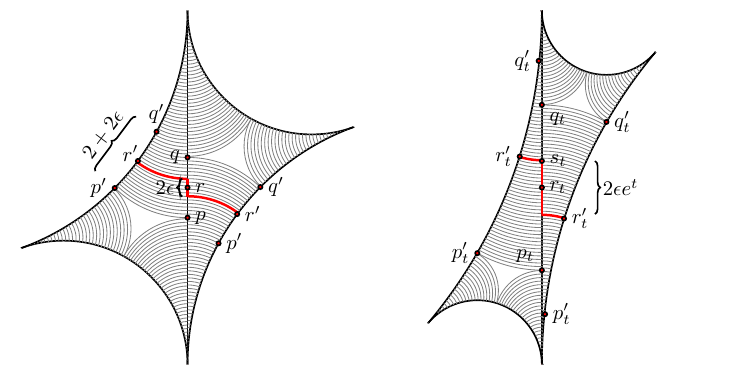} 
  }
  \end{picture}
  \caption{The annuli $A_0$ and $A_t$ are obtained by gluing two ideal
    triangles. The \emph{shifts} for $A_t$ are $e^t$--times those of
  $A_0$. }
  \label{Fig:Shear}
  \end{figure}
  
  Let $r$ and $r'$ be the midpoints of $[p,q]$ and $[p',q']$. There is an
  isometry of $A_0$ that switches the two triangles and fixes $r$ and $r'$.
  Hence, if $\gamma$ is the core curve of the annulus $A_0$, the geodesic
  representative of $\gamma$, which is unique and is fixed by this
  isometry, passes through points $r$ and $r'$.

   Define $X_t=\str(X_0,\lambda,t)$.  We give an estimate for the hyperbolic 
  length  of $\gamma$ at $X_t$ for $t\in \RR_+$. 
  \begin{claim*} 
    For $t\in \RR_+$, we have $\ell_{X_t}(\gamma)\emul \epsilon \, e^t+e^{-e^t}$.
  \end{claim*}
  \begin{proof}
  Let $A_t$ be the annulus obtained by gluing two triangles when the shifts
  are $2 e^t$ and $(2+2\ep)e^t$ and let $p_t, q_t, r_t$ and $r'_t$ be
  points that are defined similar to $p, q, r,$ and $r'$ in $A_0$. The
  geodesic representative of the core curve of $A_t$, which we still denote
  by $\gamma$, passes through the points $r_t$ and $r'_t$. Denote the
  length of the segment $[r_t, r'_t]$ by $d(r_t, r_t')$. Then 
  \[ \ell_{X_t}(\gamma)= 2 d(r_t,r'_t).\] 
  To estimate $d(r_t,r'_t)$ we work in one of the ideal triangles.
  Let $s_t\in [p_t,q_t]$ be the point which is on 
  the  same horocycle as $r'_t$.  
  Consider the triangle $[r_t,r'_t]\cup[r'_t,s_t]\cup[s_t,r_t]$. 
  By the triangle inequality we have 
  \[
  d(r_t,r'_t)\leq d(r_t,s_t)+d(s_t,r'_t) \leq 2 \max \big\{ d(r_t,s_t),
  d(s_t,r'_t) \big\}.
  \]
  On the other hand, the angle between the segments $[r'_t,s_t]$ and
  $[s_t,r_t]$ is at least $\frac \pi 2$, which implies that the side
  $[r_t,r'_t]$ is the largest of the triangle. Hence, we also have 
  \[
  d(r_t,r'_t)\geq  \max \big\{d(r_t,s_t),d(s_t,r'_t) \big\}.
  \] 
  That is, up to a multiplicative error of at most $4$, the length
  $\ell_t(\gamma)$ is $d(r_t,s_t)+d(s_t,r'_t)$. The distance $d(s_t,r'_t)$
  is asymptotically (as $t\to +\infty$) equal to the length of the horocycle between $s_t$ and
  $r'_t$. It is straightforward to see that since $d(p_t,s_t)=e^t(1+\epsilon)$, the length of the 
  horocycle is  $e^{-e^t(1+\epsilon)}$. Also  $d(r_t,s_t)=\epsilon e^t$, and we have 
  \[
  \ell_{X_t}(\gamma)\emul \epsilon e^t+e^{-e^t(1+\epsilon)}.
  \]
  The second term in the sum can be replaced with $e^{-e^t}$ without
  increasing the multiplicative error by much. This is true because the
  first term  $\epsilon e^t$ in the sum is bigger than the second term when
  $\epsilon e^t$ is bigger than $1$. This proves the claim.
  \end{proof}
 
  We can now approximate the minimum of $\ell_{X_t}(\gamma)$ for $t\in \RR_+$. At $t=0$ and
  $t=\log \frac{1}{\epsilon}$ the length of $\gamma$ is basically $1$. If
  $\epsilon$ is  small enough, there is $t_0>0$ such that $\epsilon
  e^{t_0}=e^{-e^{t_0}}$. Then we have 
  \[
  \ell_{X_t}(\gamma)\emul \begin{cases}
  e^{-e^t}, & \quad  t<t_0,\\ 
   \epsilon e^t, & \quad t>t_0.
  \end{cases}
  \] 
  This means in particular that the length of $\gamma$ decreases
  super-exponentially fast, reaches its minimum, and grows back up
  exponentially fast. We will not compute the exact value of $t_0$, but if
  we take log twice we see that $t_0\eadd \log \log\frac 1 \epsilon$, with
  additive error at most $\log 2$. Then $\ell_{t_0}(\gamma)\emul \epsilon
  \log\frac{1}{\epsilon}$ and this is, up to a multiplicative error, the
  minimum of $\ell_{X_t}(\gamma)$. 
 
  The reason the curve $\gamma$ gets short and then long again is because
  it is more efficient to twist around $\gamma$ when $\gamma$ is short. To
  see this, we estimate the relative twisting of  $X_0$ and $X_t$ around
  $\gamma$ at $t=\log{\frac{1}{\epsilon}}$, that is when the length of
  $\gamma$ grew back to approximately $1$.  The lamination $\lambda$ is
  nearly perpendicular to $\gamma$ at $X_0$ and does not twist around it.
  Hence, we need to compute how many times it twists around $\gamma$ in
  $X_t$. 
  
   For a fixed $t>0$, choose lifts $\tilde\gamma$ and  $\tilde\lambda$ to
   $\HH$  of the geodesic representative of $\gamma$ and of the leaf of
   $\lambda$ containing $[p_t,q_t]$, such that $\tilde\gamma$ and
   $\tilde\lambda$ intersect. Let $\ell$ be the length of the orthogonal
   projection of $\tilde\lambda$ to $\tilde\gamma$. Then (see \cite[\S
   3]{Min96a})
   $$\twist_\gamma(\lambda,X_t)\eadd\frac{\ell}{\ell_{X_t}(\gamma)}.$$ To
   find $\ell$, we note that $\cosh{\ell/2}=\frac{1}{\sin \alpha}$, where
   $\alpha$  is the angle between  $\tilde\gamma$ and $\tilde\lambda$. In
   the triangle $[r_t,r'_t]\cup[r'_t,s_t]\cup[s_t,r_t]$, $\alpha$ is the
   angle between segments $[s_t,r_t]$ and $[r_t,r'_t]$. Let $\beta$ be the
   angle between $[r'_t,s_t]$ and $[s_t,r_t]$. Since $\beta$ is
   asymptotically $\pi/2$, by the hyperbolic sine rule, we have  
  \[ 
  \cosh{\ell/2}\emul \frac{\sinh{d(r'_t,r_t)}}{\sinh{d(r'_t,s_t)}}.
  \]
  Assuming  $t=\log\frac 1 \epsilon$,   we have $\sinh{d(r'_t,r_t)}\emul 1$ 
  and $\sinh{d(r'_t,s_t)}\emul e^{-\frac 1 \epsilon}$
  which implies $\ell\eadd \frac 2 \epsilon$. Hence
  \[
  \twist_{\gamma} (\lambda,X_{\log \frac  1 \epsilon})\emul \frac 1 \epsilon.
  \]
  
  To summarize, the surface $X_{\log \frac  1 \epsilon}$ is close to
  $D_\gamma^n(X_0)$, where $D_\gamma$ is a Dehn twist around $\gamma$ and
  $n \emul \frac 1 \ep$.  The stretch path $\str(X_0, \lambda, t)$
  from $X_0$ to $X_{\log \frac  1 \epsilon}$ changes only an annular neighborhood of
  $\gamma$, first decreasing the length of $\gamma$ super-exponentially
  fast to $\ep \log \frac 1 \ep$ and then increasing it exponential fast.
  In fact, further analysis shows that essentially all the twisting is done
  near the time $t_0$ when the length of $\gamma$ is minimum. 
  
  \subsection{Shadow map} \label{Sec:ShadowMap}
  
  For any $X \in \T$, the set of systoles on $X$ has uniformly bounded
  diameter in \cc. We will call the coarse map $\pi \colon \T \to \cc$
  sending $X$ to the set of systoles on $X$ the \emph{shadow map}. The
  following lemma shows that the shadow map is Lipschitz. 
  
  \begin{lemma}
     
    The shadow map $\pi \colon \T \to \cc$ satisfies for all $X,Y\in \T$
     \[ \dS \big( \pi(X), \pi(Y)
    \big) \prec \dL(X,Y). \]

  \end{lemma}
  
  \begin{proof}
     
    Let $X, Y \in \T$ and let $K= L(X,Y)$. Let $\alpha$ be a systole on $X$
    and let $\beta$ be a systole on $Y$. Recall that $\epsilon_B$ is the
    Bers constant and $\delta_B$ its dual constant defined in
    \subsecref{collar}. We have $\ell_X(\alpha) \le \ep_B$ and
    $\ell_Y(\beta) \le \ep_B$. Now 
    \[ 
      \I(\alpha,\beta) \le \frac{\ell_Y(\alpha)}{\delta_B} \le \frac{K
      \ell_X(\alpha)}{\delta_B} \le K \frac{\ep_B}{\delta_B}.
    \]
    Therefore  \eqnref{Intersection} implies $\dS(\alpha,\beta) \prec \log
    K = \dL(X,Y)$. \qedhere 

  \end{proof}
  
  For simplicity, we will often write $\dS(X,Y) := \dS\big(\pi(X), \pi(Y)\big)$.  

\section{Hyperbolic geometry}

  \label{Sec:Hyperbolic}

  In this section, we establish some basic properties of the hyperbolic
  plane $\HH$ and hyperbolic surfaces. Many of these results are known in
  spirit, but to our knowledge the exact statements do not directly follow
  from what is written in the literature. 

  Recall that $\HH$ is Gromov hyperbolic, that is, there is a constant
  $\delta_\HH$ such that all triangles in $\HH$ are $\delta_\HH$ slim:
  every edge of a triangle is contained in a $\delta_\HH$--neighborhood of
  the union of the other two edges. 
 
  \subsection{Geodesic arcs on hyperbolic surfaces}
  
  Let $\alpha$ be a simple closed geodesic on a hyperbolic surface $X$ and
  let $U(\alpha)$ be the standard collar of $\alpha$. When $\omega$ is a
  geodesic segment contained in $U(\alpha)$ with endpoints $p$ and $p'$, we
  denote the distance between $p$ and $p'$ in $U(\alpha)$ by
  $d_{U(\alpha)}(\omega)$. The following lemma can be read as saying that
  all the twisting around a  curve $\alpha$ takes place in $U(\alpha)$. 
   
  \begin{lemma} \label{Lem:Arc-In-Pants}

    Let $P$ be a pair of pants in a hyperbolic surface $X$ with geodesic
    boundary lengths  less than $\ep_B$. For each connected component
    $\alpha \subset \partial P$, there is an arc $\tau_\alpha$ in
    $U(\alpha)$ perpendicular to $\alpha$ such that the following holds.
    Any finite sub-arc  $\omega$ of a simple complete geodesic $\lambda$
    that is contained in $P$ can be subdivided into three pieces
      \[
      \omega = \omega_\alpha \cup \omega_0 \cup \omega_\beta
      \]
    such that 
    \begin{enumerate}[(a)]
      
      \item The interior of $\omega_0$ is disjoint from every $U(\gamma)$,
      for $\gamma \subset \partial P$, and $\ell_X(\omega_0) = O(1)$. 
               
      \item The  segment $\omega_\alpha$, $\alpha \subset \partial P$, is
      contained in  $U(\alpha)$ and intersects any curve in $U(\alpha)$
      that is equidistant to $\alpha$ at most once. That is, as one travels
      along $\omega_\alpha$, the distance to $\alpha$ changes
      monotonically. Furthermore, 
      \[
        \ell_X(\omega_\alpha) \eadd \I( \tau_\alpha, \omega) \ell(\alpha) +
        d_{U(\alpha)}(\omega_\alpha),
      \]
      
      \item  The same holds for $\omega_\beta$ ($\alpha$ and $\beta$ may be 
      the same curve). 
      
    \end{enumerate}

  \end{lemma}
  
  \begin{proof}
      
    Note that if $\omega$ intersects $U(\alpha)$, then by \lemref{Collar},
    $\lambda$ either intersects $\alpha$ or spirals towards $\alpha$. This
    implies that $\omega$ intersects at most two standard collars, say
    $U(\alpha)$ and $U(\beta)$. We allow the possibility that $\alpha =
    \beta$. Let $\omega_\alpha = \omega \cap U(\alpha)$, $\omega_\beta =
    \omega \cap U(\beta)$, and $\omega_0$ be the remaining middle segment.
    In $P$, there is a unique geodesic segment $\eta$ perpendicular to
    $\alpha$ and $\beta$. Set $\tau_\alpha= \eta\cap U(\alpha)$ and
    $\tau_\beta= \eta\cap U(\beta)$. If $\alpha=\beta$, then $\eta$ is the
    unique simple segment intersecting $U(\alpha)$ twice and perpendicular
    to $\alpha$, and $\tau_\alpha$ and $\tau_\beta$ are the two components
    of $\eta$ in $U(\alpha)$. In the universal cover, a lift $\tomega$ of
    $\omega$ is in a $2\delta_\HH$--neighborhood of the union of a lift
    $\talpha$ of $\alpha$, a lift $\tbeta$ of $\beta$ and a lift $\teta$ of
    $\eta$.
 
    In fact, a point in $\tomega$ is either in a $\delta_B$--neighborhood
    of $\talpha \cup \tbeta$, where $\delta_B$ is the dual constant to
    $\ep_B$, or is uniformly close to $\teta$. This fact has two
    consequences. First, because $U(\alpha)$ and $U(\beta)$ have
    thicknesses bigger than $\delta_B$, the lift $\tomega'$ of $\omega'$ is
    contained in a uniform bounded neighborhood of $\teta$ and hence, the
    length of $\omega'$ is comparable to the length of $\eta$ outside of
    $U(\alpha)$ and $U(\beta)$, which is uniformly bounded.    
    
    Secondly, $\tomega_\alpha$ is in a $\delta_\HH$--neighborhood of
    $\tilde \tau_\alpha$ and $\talpha$. The portion that is in the
    neighborhood of $\tilde \tau_\alpha$ has a length that is, up to an
    additive error, equal to $d_{U(\alpha)}(\omega_\alpha)$. The portion
    that is in the neighborhood of $\talpha$ has a length that is (up to an
    additive error) equal to
    \[
      \I(\tau_\alpha, \omega_\alpha) \ell(\alpha).
    \]
    The formula in part $(b)$ follows from adding these two estimates. 
    
    The only remaining point is that, in the above, the choice of
    $\tau_\alpha$ depends on $\beta$ is. However, we observe that the
    choice of $\tau_\alpha$ is not important and for any other segment
    $\tau_\alpha'$ perpendicular to $\alpha$ that spans the width of
    $U(\alpha)$, we have
    \[
     \I(\tau_\alpha, \omega_\alpha) \eadd \I(\tau_\alpha', \omega_\alpha).
    \]
    This finishes the proof. 
  \qedhere  

  \end{proof}

   Let $X \in \T$ and $P$ be a pair of pants in $X$ with boundary lengths
   at most $\ep_B$.  Roughly speaking, the following technical lemma states
   that if  a subsegment  $\omega$ of $\lambda$ intersects a closed 
   (non geodesic) curve $\gamma$ enough times and the consecutive intersection 
   points are far enough apart, then $\omega$ cannot be contained in $P$. 
    
  \begin{lemma} \label{Lem:Arc-In-Pants2}

    There exist $D_0 > 0$ and $K_0 > 0 $ with the following property. Let
    $\gamma$ be a simple closed curve ($\gamma$ may not be a geodesic) that
    intersects $\partial P$ and let $\omega$ be a simple geodesic in $P$
    where the consecutive intersections of $\omega$ with $\gamma$ are at
    least $D_0$ apart in $\omega$. Let $\bomega$ be a sub-arc of $\omega$
    with endpoints on $\gamma$, and let $\omega_1$ and $\omega_2$ be the
    connected components of $\omega \setminus \bomega$. Then at least one
    of $\bomega$, $\omega_1$, or $\omega_2$ has length bounded by $K_0 \,
    (\ell_X(\gamma)+1)$.

  \end{lemma} 

  \begin{proof}
 
    For $i= 1,2$, let $p_i$ be the endpoints of $\bomega$ that is also an
    endpoint of $\omega_i$. Let $\alpha_i$ be the boundary curve of $P$
    with $p_i \in U(\alpha_i)$, where $U(\alpha_i)$ is the standard collar
    of $\alpha_i$. If $p_i$ does not belong in any
    of the collar neighborhoods of a curve in $\partial P$, then we choose
    $\alpha_i$ arbitrarily. For simplicity, denote $U(\alpha_i)$ by $U_i$
    and let $\bomega_i$ be the component of $\bomega \cap U_i$ with
    endpoint $p_i$ ($\bomega_i$ may be empty). It is enough to show that,
    for $i=1,2$, we have either 
    \begin{equation} \label{Eqn:Dichotomy} 
      \ell_X(\bomega_i) \prec  \ell_X(\gamma) \quad \text{or}\quad 
      \ell_X(\omega_i) \prec   \ell_X(\gamma).
    \end{equation}
    This is because if $\omega_1$ and $\omega_2$ are both very long, then
    \eqnref{Dichotomy} implies that $\bomega_1$ and $\bomega_2$ both have
    length bounded by $\ell_X(\gamma)$ up to a small error. Since the
    middle part of $\bomega$ has bounded length (\lemref{Arc-In-Pants}),
    this implies the desired upper bound for the length of $\bomega$. 
    
    We now prove \eqnref{Dichotomy}. The point $p_i$ subdivides 
    $U_i$ into to two sets $V_i$ and  $W_i$ where $V_i$ and $W_i$ are regular 
    annuli (their boundaries are equidistance curves to $\alpha$) with disjoint
    interiors and $p_i$ is on the common boundary of $V_i$ and $W_i$. By
    part (c) of \lemref{Arc-In-Pants}, one of these annuli contains
    $\omega_i$ and the other contains $\bomega_i$. Also, since $\gamma$
    passes through $p_i$, it intersects both boundaries of either $V_i$ or
    $W_i$. 
    
    Let $V$ be either  $V_i$ or $W_i$ such that $\gamma$ intersect both
    boundaries of $V$ and let $\eta$ be either $\omega_i$ or $\bomega_i$
    that is contained in $V$. We want to show
    \[
      \ell_X(\eta) \prec  \ell_X(\gamma),
    \]
    which is equivalent to \eqnref{Dichotomy}.
    
    From \lemref{Arc-In-Pants} we have
    \begin{equation}\label{Eqn:OmegaLength}
      \ell_X(\eta)\eadd
      \I(\eta,\tau_\beta)\ell_X(\alpha)+d_V(\eta).
    \end{equation}
    Let $\sigma$ be the geodesic representative of the sub-arc of $\gamma$
    connecting the boundaries of $V$. Then
    \begin{equation}\label{Eqn:Sigma}
      d(\eta ,V)\ladd \ell_X(\sigma)\leq \ell_X(\gamma). 
    \end{equation}
    The intersection numbers between arcs in an annulus satisfy the
    triangle inequality up to a small additive error. Hence, 
    \begin{equation}\label{Eqn:OmegaTau}
      \I(\eta,\tau_\alpha)\ladd
      \I(\eta,\sigma)+\I(\sigma,\tau_\alpha)\leq
      \I(\eta,\gamma)+\I(\sigma,\tau_\alpha). 
    \end{equation}
    
    Let $D_0  > \ep_B$ be any constant. By assumption, consecutive
    intersections of $\omega$ with $\gamma$ are at least $D_0$ apart in
    $\omega$. We have $\I(\eta,\gamma) \leq \frac {1}{D_0} \ell_X(\eta)$.
    Also, $\I(\sigma, \tau_\alpha) \le
    \frac{\ell_X(\sigma)}{\ell_X(\alpha)}+1$. These facts and
    \eqnref{OmegaTau} imply 
    \[
      \I(\eta,\tau_\alpha) \ladd \frac {1}{D_0} \ell_X(\eta) +
      \frac{\ell_X(\sigma)}{\ell_X(\alpha)}.
    \]
    Combining this with \eqnref{OmegaLength} and \eqnref{Sigma} we have
    \begin{align*}
      \ell_X(\eta) & \ladd
      \left( \frac  {1}{D_0} \ell_X(\eta) + \frac{\ell_X(\sigma)}{\ell_X(\alpha)}\right) 
      \ell_X(\alpha) + \ell_X(\gamma)  \\
      & \leq 
      \frac{\ep_B}{D_0}\ell_X(\eta)+2\ell_X(\gamma). 
    \end{align*}
    That is,     
    \[
      \left( 1 - \frac {\ep_B}{D_0} \right)  \ell_X(\eta ) \ladd 2
      \ell_X(\gamma). 
    \]
    The constant $1-\frac{\ep_B}{D_0}$ is positive, since $D_0 > \ep_B$. So
    taking $K_0$ sufficiently larger than $2 \left(1-\frac{\ep_B}{D_0}
    \right)^{-1}$ finishes the proof.  \qedhere

  \end{proof} 
 
  \begin{lemma} \label{Lem:Homotopic}

    Let $X\in \T$ and let $\alpha$ be a simple closed geodesic in the
    hyperbolic metric of $X$.  Let $\omega$ be a simple geodesic arc in
    $X$.  If  $$\I(\omega,\alpha)\geq 3,$$ then  any  curve $\gamma$ in the
    homotopy class of $\alpha$ that is disjoint from $\alpha$ intersects
    $\omega$ at least once.

  \end{lemma}
  \begin{remark}
     Note that the statement is sharp in the sense that  if $\gamma$
      is not disjoint from $\alpha$ or if $\omega$ intersects $\alpha$ 
      less than three times, $\gamma$ can be disjoint from $\omega$.

  \end{remark}
  \begin{proof}

    Curves $\alpha$ and $\gamma$ as above bound an annulus $A$ in $X$.
    Suppose $\omega$ intersects $\alpha$ at points $p_1, p_2$ and $p_3$,
    the points being ordered by their appearances along $\omega$. For
    $i=1,2$, let $[p_i,p_{i+1}]$ be the segment of $\omega$ between $p_i$
    and $p_{i+1}$. Since $\alpha$ and $\omega$ are geodesics, their
    segments cannot form bigons. Therefore one of the segments $[p_1,p_2]$
    and $[p_2,p_3]$ intersects the interior, and therefore both boundary
    components of $A$. \qedhere

  \end{proof}

  \subsection{}

  We now prove several useful facts about geodesics in the hyperbolic plane.   
  
  \begin{proposition} \label{Prop:Sandwich}
    
    Let $\phi \colon \HH \to \HH$ be a hyperbolic isometry with axis
    $\gamma$ and translation length $\ell(\phi) \le \ep$, for some $\ep>0$. 
    Suppose $\alpha$ and $\beta$ are two geodesic lines such that the
    following two conditions hold: 
    \begin{itemize}

      \item[(1)] $\alpha$ and $\beta$ intersect $\gamma$ at the points $a$ and
      $b$, respectively, with $\dH(a,b) \le \ep$.

      \item[(2)] Let $c$ be the point on $\beta$ which is closest to $\alpha$
      (the intersection point between $\alpha$ and $\beta$ if they
      intersect). Assume $\dH(c,b) \ge 2 \ep$. 

    \end{itemize}
  
    Fix an endpoint $\beta_+$ of $\beta$ and let $\alpha_+$ be the endpoint
    of $\alpha$ that is on the same side of $\gamma$ as $\beta_+$. Then,
    there exists $k\in \ZZ$ such that $\alpha_+$ is between $\beta_+$ and
    $\phi^k(\beta_+)$ and $|k| \ell(\phi) \le 4\ep+3$. 
    
  \end{proposition}
  
  \begin{figure}[ht]
  \setlength{\unitlength}{0.01\linewidth}
  \begin{picture}(100, 40)
  \put(25,-3){
  \includegraphics{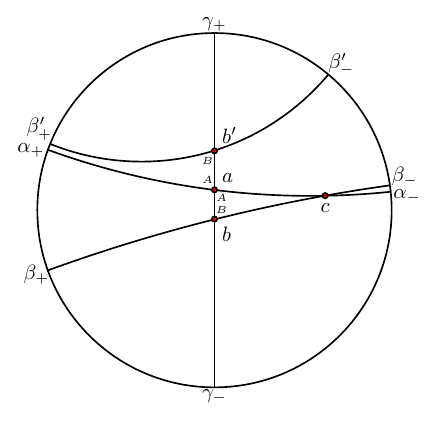}
  }
  \end{picture}
  \caption{}
  \label{Fig:Sandwich}
  \end{figure} 
 
  \begin{proof}

    We assume $\beta_+$ is the endpoint of the ray $\overrightarrow{cb}$
    (see \figref{Sandwich}). The case when $\beta_+$ is the endpoint of
    $\overrightarrow{bc}$ is similar. 
    
    By exchanging $\phi$ with $\phi^{-1}$ if necessary we may assume
    $\alpha_+$ is between $\beta_+$ and the attracting fixed point
    $\gamma_+$ of $\phi$. Let $k$ be a positive integer with 
    \[
      3 \ep + 3 \le k\ell(\phi) \le 4\ep+3.
    \]
    Such a $k$ exists since $\ell(\phi) \leq \ep$. Let $\beta' =
    \phi^k(\beta)$, $b' = \phi^k(b)$ and $\beta_+' = \phi^k(\beta_+)$. 
    
    Since $k \ell(\phi) \geq \ep$, we know $a$ is between $b$ and $b'$.
    To show $\alpha_+$ is between $\beta_+$ and $\beta'_+$ we need to show
    that the rays $\overrightarrow{a\alpha_+}$ and
    $\overrightarrow{b'\beta'_+}$ do not intersect. Note that
    \[
      \dH(b',a) = \dH(b',b) -\dH(a,b) \geq (3\ep + 3) - \ep = 2 \ep +
      3 > \dH(a,b).
    \]
    Let $A = \angle \alpha_-a \gamma_- = \angle \alpha_+ a \gamma_+$ and $B
    = \angle \beta_-b\gamma_+ = \angle \beta_+' b' \gamma_-$. If $\alpha$
    and $\beta$ are disjoint then $\overrightarrow{a\alpha_+}$ and
    $\overrightarrow{b'\beta'_+}$ are also disjoint and we are done. Hence,
    we can assume $\alpha$ and $\beta$ intersect at $c$. 
    
    Using the law of cosines, $\overrightarrow{a\alpha_+}$ and
    $\overrightarrow{b'\beta'_+}$ do not intersect if  
    \begin{equation} \label{Eqn:Cosine}
      \sin A \sin B \cosh{\dH(a,b')}- \cos A \cos B \ge \cos 0=1. 
    \end{equation}
    Since $\dH(a,b') \ge 2\ep+3$, 
    \[
      \cosh{\dH(a,b')} \ge \cosh \big( 2 \ep +3 \big)
      >\frac {e^{2 \ep + 3}}2.
    \]
    Hence, by \eqnref{Cosine}, it suffices to show that
    \begin{equation} \label{Eqn:LargeK}
      \frac {e^{2 \ep + 3}}2 \geq  \frac{1+\cos A \cos B } {\sin A \sin B}.
    \end{equation}
    
    Before starting the calculations, we make an elementary observation.
    For any $y>0$, the function \[ f(x) = \frac{\sinh(x+y)}{\sinh x}\] is 
    decreasing and $f(x) \le 2e^y$ for all $x \ge y$.  This is because 
    $f'(x) = \frac{-\sinh(y)}{\sinh^2(x)}<0$.  Hence, for all $x \ge y$,
    \[
      f(x) \le f(y) = \frac{\sinh 2y}{\sinh y} = 2 \cosh y \le 2e^y.
    \]

    We argue in 3 cases. Suppose $A > \pi/2$. Since $\dH(c,a) \ge \dH(c,b)
    - \dH(a,b) \ge \ep$, we have
    \begin{align*}
      \frac{1+\cos A \cos B }{\sin A \sin B}
      &=
      \frac{1-\cos (\pi-A) \cos B}{\sin A \sin B}
      \le
      \frac{\sin^2 \big( \max\{A,B\} \big)}
      {\sin^2 \big( \min\{A,B\} \big)} \\
      &=
      \frac{\sin^2 A}{\sin^2B} 
      =
      \frac{\sinh^2 \dH(c,b)}{\sinh^2 \dH(c,a)} 
      \le 
      \frac{\sinh^2 \big(\dH(c,a)+\ep \big)}{\sinh^2 \dH(c,a) } \le 4e^{2\ep} 
    \end{align*}
    Similarly, if $B \ge \pi/2$ instead,  then 
    \[
      \frac{1+\cos A \cos B }{\sin A \sin B} \le \frac{\sinh^2 \big(
      \dH(c,b)+\ell(\phi) \big)}{\sinh^2 \big( \dH(c,b) \big)} \le
      4e^{2\ep}.
    \] 

    In the case that both $A$ and $B$ are at most $\pi/2$, let $w$ be the
    point on the segment $\overline{ab}$ which is the foot of the
    perpendicular from $c$ to $\overline{ab}$. We have 
    \begin{align*}
      \frac{1+\cos A \cos B }{\sin A \sin B} 
      &\le 
      \frac{2}{\sin A \sin B} 
      =  
      2 \frac{\sinh \dH(c,a)}{\sinh \dH(c,w)} \frac{\sinh
      \dH(c,b)}{\sinh \dH(c,w)} \\ 
      & \le 
      2 \frac{\sinh^2 \big( \dH(c,w) + \ell(\phi) \big)}{\sinh^2 \dH(c,w)} 
      \le 8e^{2\ep}.
    \end{align*}
    But $8e^{2\ep} < \frac {e^{2 \ep + 3}}2$ and we are done. \qedhere
    
  \end{proof}
  
  We need some definitions for the next two lemmas. Let $\gamma$ be a
  geodesic in $\HH$ with endpoints $\gamma_+$ and $\gamma_-$. Fix a
  $\delta$--neighborhood $U$ of $\gamma$ and let $p$ be any point on the
  boundary of $U$. The geodesic through $p$ with endpoint $\gamma_+$ and
  the geodesic through $p$ with endpoint $\gamma_-$ together subdivide
  $\HH$ into four quadrants. The quadrant disjoint from the interior of $U$
  will be called the \emph{upper quadrant} at $p$, and the quadrant
  diametrically opposite will be called the \emph{lower quadrant} at $p$. 
  
  \begin{lemma}\label{Lem:Quadrant}
    
    For every $\delta_0>0$ and $M>0$, there is $d_0>0$ such that the
    following holds (see left side of \figref{Quadrant}). Fix a geodesic
    $\gamma$ in $\HH$ and let $U$  be the $\delta$-neighborhood of $\gamma$
    with $\delta\geq \delta_0$. Let $p$ and $q$ be points on the same
    boundary component of $U$. Suppose $\dH(p,q)\geq M$. Then any point in
    the upper quadrant $\calQ_p$ at $p$ is  at least $d_0$ away from any
    point in the lower quadrant $\calQ_q$ at $q$.
    
  \end{lemma}
    
  \begin{proof}

    \begin{figure}[htp!]
    \setlength{\unitlength}{0.01\linewidth}
    \begin{picture}(100, 44)
    \put(1,-2){
    \includegraphics{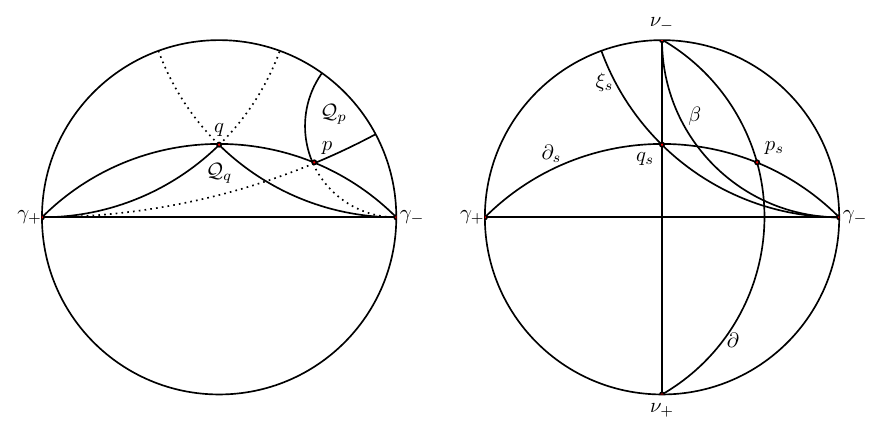} 
    }
    \end{picture}
    \caption{There is a lower bound on the distance between any point in
    $\calQ_p$ and any point in $\calQ_q$.}
    \label{Fig:Quadrant}
    \end{figure} 

    It suffices to find a lower bound for $\dH(p,\xi)$, where $\xi$ is the
    closest boundary component of $\calQ_q$ to $p$. Assume $\xi$ has the
    endpoint $\gamma_-$. Note that $\dH(p,\xi)$ increases with $\dH(p,q)$,
    since any geodesic from $\gamma_-$ that intersects the boundary of $U$
    between $p$ and $q$ separates $p$ from $\xi$.  
    
    Now fix $\dH(p,q) = M$ and see the right side of \figref{Quadrant} for
    the rest of the proof. Fix a perpendicular geodesic $\nu$ to $\gamma$
    and let $\partial$ be a boundary component of the $M$--neighborhood of
    $\nu$. Fix an endpoint $\nu_-$ of $\nu$ and let $\partial_s$ be a
    boundary curve of the $s$--neighborhood  of $\gamma$ contained in the
    complement $\HH \setminus \gamma$ determined by $\nu_-$. Let $q_s$ be
    the intersection of $\nu$ and $\partial_s$ and let $p_s$ be
    intersection of $\partial$ and $\partial_s$. Let $\xi_s$ be the
    geodesic through $q_s$ with endpoint $\gamma_-$ and let $\beta$ be the
    geodesic connecting $\gamma_-$ and $\nu_-$. Since $\beta$ is asymptotic
    to $\nu$, $\dH(p_s,\beta)$ is increasing as a function of $s$, and for
    $s$ big enough, $\beta$ separates $p_s$ from $\xi_s$. Therefore, there
    exist $d_1>0$ and $s_1 \ge \delta_0$ such that $\dH(p_s,\xi_s)>
    \dH(p_s,\beta)>d_1>0$ for all $s\geq s_1$. Finally, since
    $\dH(p_s,\xi_s)$ is continuous as a function of $s$ and is only zero
    when $s=0$, it is bounded away from zero on the segment
    $[\delta_0,s_1]$. This finishes the proof. \qedhere

  \end{proof} 
   
  \begin{lemma} \label{Lem:Diverge}
     
    Given $\delta_0>0$ and $M>0$, the constant $d_0$ of \lemref{Quadrant}
    also satisfies the following property. Let $\gamma$ be a geodesic in
    $\HH$. Let $U$ be the $\delta$--neighborhood of $\gamma$ with
    $\delta\geq \delta_0$. Let $\omega_1$ be a geodesic intersecting
    $\gamma$ and let $\omega_2$ be the image of $\omega_1$ under an
    isometry fixing $\gamma$. Let $p_i$ be the intersection of $\omega_i$
    with a fixed boundary component of $U$. Let $\omega_1^+$ be the
    component of $\omega_1$ contained in the upper quadrant at $p_1$.
    Suppose $d_\HH(p_1, p_2) \geq M$. Then $d_\HH(\omega_1^+, \omega_2)
    \geq d_0$.
              
  \end{lemma} 
          
  \begin{proof} 
    
    Let $\omega_2^+$ be the component of $\omega_2$ contained in the upper
    quadrant at $p_2$. By \lemref{Quadrant}, it is enough to show
    $\dH(\omega_1^+,\omega_2^+) \ge d_0$. Let $r_i$ be the point on
    $\omega_i$ such that $\dH(r_1,r_2)$ realizes the distance between
    $\omega_1$ and $\omega_2$. Because of the symmetry of $\omega_1$ and
    $\omega_2$ relative to a rotation fixing $\gamma$, either both $r_i$'s
    are contained in $U$, and hence each $r_i$ is contained in the lower
    quadrant at $p_i$, or one is in the lower quadrant and the other is in
    the upper quadrant. 
    
    Identify the space of pairs of $(x,y)$, where $x\in \omega_1$ and 
    $y \in \omega_2$, with $\RR^2$. Then the function $\RR^2 \to \RR$ 
    sending $(x,y)$ to $d_\HH(x,y)$ is a convex function 
    realizing its minimum at $(r_1,r_2) \in \RR^2$. In the first
    case, since $\dH(p_i,\omega_{3-i}^+) \ge d_0$ by \lemref{Quadrant} and
    the distance between $\omega_1^+$ and $\omega_2^+$ increases from $p_1$
    and $p_2$ on, and we can conclude $\dH(\omega_1^+,\omega_2^+) \ge d_0$. 
    In the second case, invoking \lemref{Quadrant} again implies $\dH(r_1, r_2)
    \geq d_0$, hence $\dH(\omega_1^+,\omega_2^+) \ge d_0$ by minimality of
    $\dH(r_1,r_2)$. \qedhere
   
  \end{proof}  

\section{A Notion of Being Sufficiently Horizontal}
  
  \label{Sec:Horizontal}

  Let $I$ be a closed connected subset of $\RR$, let $\GL \from I \to \T$
  be a Thurston geodesic and let $\lambda_{\GL}$ be its maximally stretched
  lamination. The main purpose of this section is to develop a notion of a
  closed curve $\alpha$ being sufficiently horizontal along \GL, such that
  if $\alpha$ is horizontal then it remains horizontal and its horizontal
  length grows exponentially along \GL. 

  \begin{definition} \label{Def:Horizontal}
  
    Given a curve $\alpha$, we will say $\alpha$ is
    \emph{$(n,L)$--horizontal} at $t \in I$ if there exists an
    $\ep_B$--short curve $\gamma$ on $X_t=\GL(t)$ and a leaf $\lambda$ in
    $\lambda_\GL$ such that the following statements hold (see
    \figref{Horizontal}): 

    \begin{itemize}
       
      \item[(H1)] In the universal cover $\tx_t \cong \HH$, there exists
      a collection of lifts $\{\tgamma_1, \ldots \tgamma_n\}$ of $\gamma$
      and a lift $\tlambda$ of $\lambda$ intersecting each $\tgamma_i$ at a
      point $p_i$ (the $p_i$'s are indexed by the order of their
      appearances along $\tlambda$) such that $\dH(p_i, p_{i+1}) \ge L$ for
      all $i=1, \ldots, n-1$.

      \item[(H2)] There exists a lift $\talpha$ of $\alpha$ such that
      $\talpha$ intersects $\tgamma_i$ at a point $q_i$ with $\dH(p_i,
      q_i) \le \ep_B$ for each $i$.

    \end{itemize}
    
    We will call $\gamma$ an \emph{anchor curve} for $\alpha$ and $\talpha$
    an $(n,L)$--\emph{horizontal lift} of $\alpha$.

  \end{definition}
  
  \begin{figure}[htp!] 
  \setlength{\unitlength}{0.01\linewidth}
  \begin{picture}(100, 36)
  \put(27.5,-2){
  \includegraphics{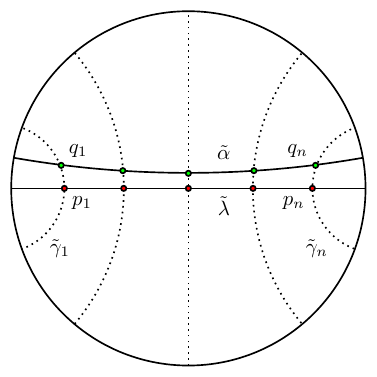}
  }
  \end{picture}
  \caption{The curve $\alpha$ is $(n,L)$--horizontal.}
  \label{Fig:Horizontal}
  \end{figure} 
  
  Set $\GL(t)=X_t$. The main result of this section is the following. 

  \begin{theorem} \label{Thm:Horizontal}
     
    There are constants $n_0$, $L_0$, and $s_0$ such that the following
    holds. Suppose a curve $\alpha$ is $(n_s,L_s)$--horizontal at $s \in I$
    with $n_s\geq n_0$ and $L_s \geq L_0$. Then
    \begin{enumerate} 
    
      \item[(I)] For any $t \ge s+s_0$, $\alpha$ is $(n_t,L_t)$--horizontal
      at $t$, with \[ n_t \gmul n_s, \quad\text{and}\quad  L_t \geq  L_s.
      \] 
      
      \item[(II)] Furthermore, for any $A$, if $\dS \big( X_s,X_t \big) \ge
      A$ then 
      \[ 
        \log \frac{n_t}{n_s} \succ A \quad\text{and}\quad L_t n_t \gmul
        e^{t-s} L_s n_s. 
      \]
    \end{enumerate}

  \end{theorem}
  
  \begin{definition}[Sufficiently Horizontal] \label{Def:SuffHorizontal}
  
    Let $(n_0,L_0)$ be the constants given by \thmref{Horizontal}. A curve
    $\alpha$ will be said to be \emph{sufficiently horizontal} at $t \in I$
    if it is $(n,L)$--horizontal for some $n \ge n_0$ and $L \ge L_0$.

  \end{definition}

  \begin{example} \label{Exa:Shear}

  \defref{Horizontal} is a bit technical and warrants some justification.
  To maintain sufficiently horizontal along a Thurston geodesic \GL, we
  require the curve $\alpha$ to fellow-travel $\lambda$ both
  \emph{geometrically} and \emph{topologically} for a long time. The
  following example illustrates why these requirements are necessary.
  Namely, we will show that the weaker version of geometric
  fellow-traveling does not always persist along a Thurston geodesic. 
  
  Referring to the example in \subsecref{Example}, for any \ep, there exists
  a Thurston geodesic $\GL(t)=X_t$ and a curve $\gamma$ such that a leaf
  $\lambda$ of $\lambda_\GL$ intersects $\gamma$ and, for $t>0$,
  $\ell_{X_t}(\gamma) \emul \ep \, e^t + 2e^{-e^t}$. Consider the following
  two points along $\GL(t)$: 
  \[X=\GL \big( \log \log(1/\ep) \big) \quad \text{and} \quad Y=\GL(\log
  1/\ep).\]
  On $Y$, let $\alpha$ be the shortest curve that intersects $\gamma$ with
  $\twist_\gamma(\alpha,Y) = 0$. It was shown in \subsecref{Example} that
  \[
    \twist_\gamma(\lambda,Y) \eadd \frac{1}{\ep}.
  \] 
  This implies $\twist_\gamma(\alpha,\lambda) \eadd 1/\ep$, so $\alpha$
  intersects $\lambda$ at an angle close to $\pi/2$ in $Y$. Furthermore,
  since $\ell_Y(\gamma) \emul 1$, we have $\ell_Y(\alpha) = O(1)$.  That
  is, every large enough segment of the any lift of $\alpha$ to the
  universal cover $\ty$ intersects a lift of $\lambda$ at near right angle.
  Therefore, in $\ty$, no lift of $\alpha$ will fellow travel any lift of
  $\lambda$. 
  
  On the other hand, it was also shown in \subsecref{Example}  that  
  \[
   \ell_X(\gamma) \emul \ep \log (1/\ep).
   \] 
  For a given $L$, we can choose \ep small enough such that the collar
  neighborhood of $\gamma$ in $X$ has width at least $L$. Since $\alpha$
  and $\lambda$ both pass through this collar, in the universal cover
  $\tx$, there exists a lift $\talpha$ and a lift $\tlambda$ that are
  $O(1)$--close for $L$--length. In other words, $\alpha$ and $\lambda$
  fellow-travel in $\tx$ but they do not in $\ty$. 

  \end{example}

  \begin{remark}
  
    For a given $t \in I$, it is possible that there are no sufficiently
    horizontal curves at $t$. For instance, when the stump of $\lambda_\GL$
    is a curve that does not intersect any $\ep_B$--short curve. But this
    is the only problem, since if the stump of $\lambda_\GL$ intersects an
    $\ep_B$--short curve $\gamma$, then any sequence of curves converging
    to the stump will eventually be sufficiently horizontal, with anchor
    curve $\gamma$. In \secref{Shadow}, we will show that one can always
    find a sufficiently horizontal curve after moving a bounded distance in
    \cc (\propref{HorExist}).
  
  \end{remark}

  The next proposition will show that the condition (H2) of \defref{Horizontal}
  can be obtained by just assuming $\talpha$ stays $\ep_B$--close to the
  segment $[p_1,p_n]$ in $\tlambda$. A priori, even if $\talpha$ is
  $\ep_B$--close to $[p_1,p_n]$, the distance between $p_i$ and $q_i$ may
  still be large if $\tgamma_i$ is nearly parallel to $\tlambda$ or
  $\talpha$. 
    
  \begin{proposition} \label{Prop:Suf-Horizontal} 

    There are constants $n_0$ and $L_0$ such that, for any hyperbolic
    surface $X$ and constants $n \ge n_0$ and $L \ge L_0$, the following
    statement holds. Suppose $\gamma$ is an $\ep_B$--short curve in $X$,
    $\lambda$ is a complete simple geodesic in $X$, and $n$ lifts
    $\{\tgamma_i\}$ of $\tgamma$ and a lift $\tlambda$ are chosen to
    satisfy (H1). If $\alpha$ is a curve in $X$ that has a lift $\talpha$
    which stays, up to a bounded multiplicative error,  $\ep_B$--close to
    the segment $[p_1,p_n]$ in $\tgamma$, then there exist indices $l$ and
    $r$ with $r-l \gadd n$ such that $\talpha$ intersects $\tgamma_i$ at a
    point $q_i$ and $\dH(p_i, q_i) \le \ep_B$, for all $i=l, \ldots, r$.
    
  \end{proposition} 
    
  \begin{proof}

    Recall the standard collar $U(\gamma)$ is a regular neighborhood of
    $\gamma$ in $X$ that is an embedded annulus with boundary length $\emul
    1$. Let $\delta$ be the distance between $\gamma$ and the boundary of
    $U(\gamma)$.  We have (see \secref{Hyperbolic})
    \[ 
      \delta \eadd \log ( 1/ \ell_X(\gamma)). 
    \] 
     Let $L_0$ satisfy inequality 
    \begin{equation*}\label{Eqn:L0}
      \epsilon_Be^{-L_0} <\delta.
    \end{equation*}
    
    The distance between $\talpha$ and $\tlambda$ is a convex function that
    essentially either increases or decreases exponentially fast. Hence, we
    can choose a segment $\balpha$ of $\alpha$ and index $i_0$ such that
    $\balpha$ is within $\epsilon_Be^{-L}$--Hausdorff distance of
    $\blambda=[p_{i_0},p_{n-i_0}]$.  The index $i_0$ can be chosen
    independent of $n$ or $L$ because the distance between $p_i$ and
    $p_{i+1}$ is at least $L$. Let $n_0 \ge 2 i_0 +2$.
    
    Let $U_i$ be the $\delta$--neighborhood of $\tgamma_i$. By the choice
    of $\delta$, $U_i$ and $U_j$ are disjoint  for $i\neq j$.  Since $\ep_B
    e^{-L} <\delta$, endpoints of $\balpha$ are contained in $U_{i_0}$ and
    $U_{n-i_0}$. Also, for $i_0 < i < n-i_0$, $U_i$ separates $U_{i_0}$ and
    $U_{n-i_0}$ in $\HH^2$. Hence $\balpha$ intersects every $\tgamma_i$.
    Consider such an $i$ and, for simplicity, set $\tgamma = \tgamma_i$,
    $U=U_i$, $p=p_i$ and $q=q_i$.  To prove the Proposition, we need to
    show that  $\dH(p,q)\leq \ep_B$. We refer to \figref{Fellow} in the
    following.
  
    \begin{figure}[htp!]
    \setlength{\unitlength}{0.01\linewidth}
    \begin{picture}(100,38)
    \put(25,-2){ 
    \includegraphics{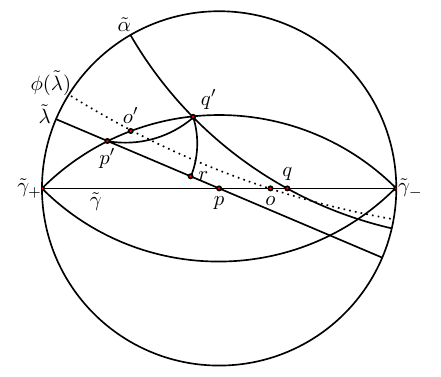} 
    }
    \end{picture}  
    \caption{If $d_\HH(p,q) \geq \ep_B$ then $q$ is far from $\tlambda$ which is
    a contradiction. }
    \label{Fig:Fellow}
    \end{figure} 

    Assume, for contradiction, that  $\dH(p,q) > \ep_B$. Let $\phi$ be a
    hyperbolic isometry with axis $\tgamma$ and translation length
    $\ell_X(\gamma)$.  Then, since $\gamma$ is $\epsilon_B$--short, up to
    replacing $\phi$ with $\phi^{-1}$, the point $o=\phi(p)$ is strictly
    between $p$ and $q$.

    Since $\tlambda$ and $\phi(\tlambda)$ are disjoint, for some boundary
    component of $U$, that we will denote by $\partial U$, the following
    holds. For  $p' = \tlambda \cap \partial U$ and $q' = \talpha \cap \partial U$,
    the point $o'=\phi(p')$ of intersection of $\partial U$ and $\phi(\tlambda)$ is 
    between $p'$ and $q'$.  The curve $\partial U$ is equidistant to $\tgamma$
    and the distance function $\dH$ is convex along this curve, therefore
    \begin{equation} \label{Eqn:p'q'}
      \dH(q',p')\geq \dH(o',p') = \ell_X(\partial U) \emul 1.
    \end{equation}
    Let $r$ be the closest point on $\tlambda$ to $q'$. The point $r$ is
    contained in one of the quadrants at $p'$, hence $\dH(q',r)\gmul 1$ by
    \lemref{Quadrant}. But for sufficiently large $L_0$, this will
    contradict that $q'$ is $\epsilon_Be^{-L_0}$ close to $\tlambda$. Hence
    $\dH(p,q)\leq \epsilon_B$ and we are done. \qedhere

    \end{proof}
     
  \begin{proof}[Proof of \thmref{Horizontal}]

    Let $n_s \ge n_0$ and $L_s \ge L_0$, $n_0$, $L_0$ to be determined
    later. Let $\alpha$ be an $(n_s,L_s)$--horizontal curve at $X_s$. As in
    the definition, we have an anchor curve $\gamma$, a lift $\talpha$ of
    $\alpha$, a lift $\tlambda$ of a leaf of $\lambda_{\GL}$, and
    $n_s$--lifts $\{ \tgamma_i \}$ of $\gamma$, such that $\dH(p_i, q_i)
    \le \ep_B$ and $\dH(p_i, p_{i+1}) \ge L_s$, where $p_i$ is the
    intersection of $\tgamma_i$ with $\tlambda$, and $q_i$ the intersection
    of $\tgamma_i$ with $\talpha$. 
   
    Throughout the proof we will add several conditions on $n_0$, $s_0$ and
    $L_0$.  Let $n_0$ and $L_0$ be at least as big as the corresponding
    constants obtained  in \propref{Suf-Horizontal}.

    Let $c$ be the point on $\tlambda$ to which $\talpha$ is closest.
    To be able to apply \propref{Sandwich} to the curve $\gamma_i$, 
    we need that $c$ has a distance of at least $2\ep_B$ from $p_i$. 
    Assuming $L_0 > 4 \ep_B$, we have $c$ is $2\ep_B$--close to at most one $p_i$.
    That is, we can choose indices $l$ and $r$, with $(l,r) = (1, n_s-1)$
    or $(l,r) =(2, n_s)$, such that $c$ has a distance at least $2\ep_B$
    from both $p_l$ and $p_r$. 
  
    \begin{figure}[ht] 
    \setlength{\unitlength}{0.01\linewidth}
    \begin{picture}(100, 39)
    \put(-1,-2){ 
    \includegraphics{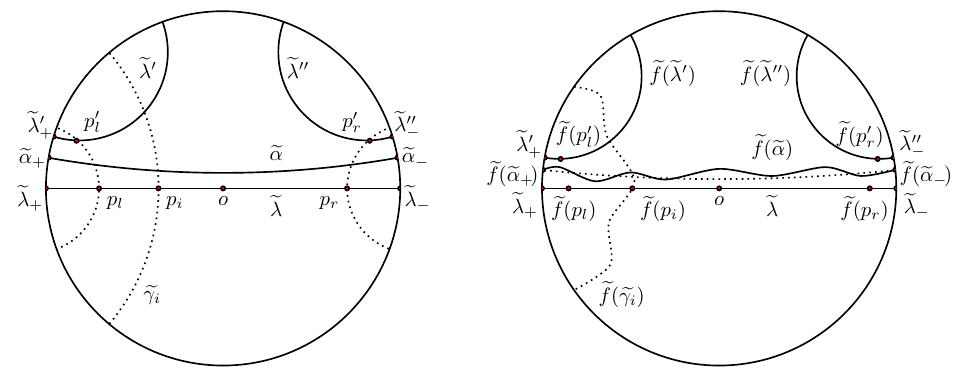} 
    }
    \end{picture}
    \caption{The endpoints of $\tf(\talpha)$ are sandwiched between the end
    points of geodesics $\tf(\tlambda)$,  $\tf(\tlambda')$ and  $\tf(\tlambda'')$.  }
    \label{Fig:Contraction}
    \end{figure} 
  
    See \figref{Contraction} for the following. Applying a M\"obius
    transformation if necessary, we can assume the center of the disk  $o$
    is the midpoint between $p_l$ and $p_r$. Let $\tlambda_+$ and
    $\tlambda_-$ be respectively the endpoints of $\tlambda$ determined by
    the rays $\overrightarrow{op_l}$ and $\overrightarrow{op_r}$. Let
    $\talpha_+$ be the endpoint of $\talpha$ closest to $\tlambda_+$. Let
    $\phi$ be the hyperbolic isometry with axis $\tgamma_l$ and translation
    length $\ell(\phi) = \ell_s(\gamma)$. Let $k$ be the constant of
    \propref{Sandwich} and let 
    \[
    \tlambda' = \phi^k(\tlambda), \qquad \tlambda_+' = \phi^k(\tlambda_+),
    \qquad\text{and}\qquad p_l' = \phi^k(p_l).
    \] 
    We have that $\talpha_+$ is sandwiched between $\tlambda_+$ and 
    $\tlambda_+'$ and $\dH(p_l,p_l') \lmul \ep_B$. Similarly, by considering the 
    hyperbolic isometry $\psi$ with axis $\tgamma_r$, we can sandwich $\talpha_-$
    between $\tlambda_-$ and $\tlambda_-''$ with $\dH(p_r,p_r') \lmul
    \ep_B$, where 
    \[
    \tlambda'' = \psi^k(\tlambda),\qquad  \tlambda_-'' =\psi^k(\tlambda_-), 
    \qquad\text{and}\qquad p_r' = \psi^k(p_r).
    \]
    Let $s_0\geq 0$, $t \ge s+s_0$ and $f \from X_s \to X_t$ be an optimal
    map, i.e.~an $e^{t-s}$--Lipschitz map. Since $\lambda_{\GL}$ 
    is in the stretch locus of $f$, there is a lift $\tf \from \tx_{s} \to \tx_t$ of $f$ 
    such that $\tf(\tlambda)=\tlambda$ and $\tf(\tlambda_{\pm}) =
    \tlambda_{\pm}$.      

    \begin{claim} \label{Clm:alpha}
      
      The geodesic representative $\alpha'$ of $\tf(\talpha)$ stays 
      $O(\ep_B)$--close to $\tlambda$ from $\tf(p_{l+1})$ and $\tf(p_{r-1})$.      

    \end{claim}
   
    \begin{proof}

      Composing with a M\"obius transformation if necessary, we may assume
      $\tf(o)=o$.   Note that $\tf(\tlambda')$ is a geodesic and
      $\tf(\talpha_+)$ is sandwiched between $\tlambda_+$ and
      $\tf(\tlambda_+')$. Similarly, $\tf(\talpha_-)$ is sandwiched between 
      $\tlambda_-$ and $\tf(\tlambda_-')$. Consider the sector $V_+$ 
      between the rays $\overrightarrow{o\, \tlambda_+}$ and 
      $\overrightarrow{o \, \tf(p_l')}$ and the sector $V_-$
      between the rays $\overrightarrow{o\, \tlambda_-}$ and 
      $\overrightarrow{o \, \tf(p_r')}$. The geodesic $\alpha'$ connecting 
      $\tf(\talpha_+)$ and $\tf(\talpha_-)$ stays in a bounded neighborhood of the 
      union $V_+$ and $V_-$. 
            
      Note that 
      \[
       d_\HH (\tf(p_l), \tf(p_{l+1}) \geq L_0 e^{t-s}
       \qquad\text{and}\qquad
       d_\HH (\tf(p_l) , \tf(p_l')) \leq e^{t-s} \ep_B. 
      \]
      Also, the distance between intersecting geodesics increases exponentially 
      fast. Hence
      \begin{align*}
      d_\HH \left(\alpha', \tf(p_{l+1})\right) 
      &\ladd e^{-d_\HH \left(\tf(p_l), \tf(p_{l+1}) \right) } d_\HH \left(\tf(p_l) , \tf(p_l')\right)\\
      & \leq e^{-L_0 \, e^{t-s}} e^{t-s} \ep_B \leq \ep_B. 
      \end{align*}
      The last inequality holds as long as $L_0 \geq 1$. 
      Similarly, $d_\HH \left(\alpha', \tf(p_{r-1})\right) \leq \ep_B$.
      \qedhere
     
    \end{proof}
   
    We next show that the projection to $X_t$ of any long enough piece of the 
    segment of $\tlambda$ between $\tf(p_l)$ and $\tf(p_r)$ intersects a lift of an
    $\epsilon_B$--short curve. 
   
    \begin{claim} \label{Clm:Intersection}
   
      For any $l \le i \le r-3$, let $\tomega$ be the arc connecting 
      $\tf(p_i)$ and $\tf(p_{i+3})$, and let $\omega$ be the projection of $\tomega$ 
      to $X_t$. Then $\omega$ intersects an $\ep_B$--short curve. 
  
    \end{claim}
   
    \begin{proof} 
     
      Recall the dual constant $\delta_B>0$ to $\ep_B$, which is a lower
      bound for the length of any curve that intersects an $\ep_B$--short
      curve. 
      
      If $\omega$ is not simple, then $\tlambda$ is a lift of a closed
      curve $\lambda$ and $\omega$ wraps around $\lambda$. By definition
      $\lambda$ intersects an $\ep_B$--short curve in $X_s$ which implies
      that $\ell_{X_s}(\lambda) \geq \delta_B$ and $\ell_{X_t}(\lambda)
      \geq \delta_Be^{t-s}$. If  $t-s \ge s_0 >
      \log\frac{\epsilon_B}{\delta_B}$, then  $\ell_{X_t}(\lambda) \geq
      \ep_B$ and $\lambda$ has to intersect an $\ep_B$--short curve in
      $X_t$. Thus, $\omega$ will also intersect an $\ep_B$--short curve. 
        
      Now assume $\omega$ is simple. Let $\gamma'=f(\gamma)$. Note that
      $\gamma'$ is not necessarily a geodesic in the metric $X_t$. If
      $\omega$ misses all the curves of length at most $\ep_B$, then it is
      contained in a pair of pants $P$ in $X_t$ with boundary lengths at
      most $\ep_B$. Since endpoints of $\omega$ lie on $\gamma'$, $\gamma'
      \cap P$ is non-empty. 
      
      First, we assume $\gamma'$ does not intersect $\partial P$.  Then
      $\gamma' \subset P$ and it is homotopic to a boundary component of
      $P$. That is, the geodesic representative $\gamma^*$ of $\gamma'$ in
      $X_s$ is $\ep_B$--short. The arc $f^{-1}(\omega)$ is a geodesic in
      $X_s$ and intersects $\gamma$ (which is a geodesic in $X_s$) at least
      $3$ times, so by \lemref{Homotopic}, $f^{-1}(\omega)$ intersects
      $f^{-1}(\gamma^*)$ at least once. This implies $\omega$ intersects
      $\gamma^*$ which proves the claim.
      
      Now assume $\gamma'$ intersects $\partial P$. For $j=0,1,2$, let
      $\omega_j$ be the sub-arc of $\omega$ coming from projecting
      $\left[\tf(p_{i+j}),\tf(p_{i+j+1})\right]$ to $X_t$. Let $D_0$ and
      $K_0$ be the constants of \lemref{Arc-In-Pants2}. Note that if a
      sub-arc of $\lambda$ intersects $\gamma'$, then the arc length
      between two consecutive intersections is at least $\delta_B
      e^{t-s}\geq \delta_Be^{s_0}$. Assuming $s_0 \geq
      \log\frac{D_0}{\delta_B}$, we have, for each $j$, $\ell_t(\omega_j)$
      is at least $L_0e^{t-s}$, while $\ell_t(\gamma')$ is at most
      $\epsilon_Be^{t-s}$. Let $L_0$ be bigger than $K_0(\epsilon_B+1)$.
      Then, by \lemref{Arc-In-Pants2}, at least one of $w_j$ has
      \[
        \ell_t(\omega_j)\leq  K_0\cdot (\ell_t(\gamma')+1)<L_0e^{t-s} 
      \]
      which is impossible and hence $\omega$ intersects one of the pants
      curves.
      \qedhere

    \end{proof}

    \clmref{Intersection} implies that  for some $\ep_B$--short curve (call it 
    $\gamma_t$) and some $n_t\gmul n_s$, the projection of 
    $\left[\tf(p_{l+1}), \tf(p_{r-1})\right]$ to $X_t$ intersect $\gamma_t$ 
    at least $n_t$ times with the  the arc length between every two
    intersection points is at least $L_t=L_se^{t-s}$. And \clmref{alpha}
    implies that there is aloft of $\alpha$ that remains $O(\ep_B)$--close
    to the segment $\left[\tf(p_{l+1}), \tf(p_{r-1})\right]$. Applying
    \propref{Suf-Horizontal} we conclude that $\alpha$ is $(n_t,L_t)$-horizontal 
    which proves part (I). 
   
    We now prove part (II) of the \thmref{Horizontal}. Suppose
    \begin{equation} \label{Eqn:distCS}
      \dS(X_s,X_t) \ge A. 
    \end{equation}
    We need to show that $n_t$ lifts of an $\ep_B$--short curve intersect the 
    segment $\left[\tf(p_{l+1}),\tf(p_{r-1})\right]$, where $\log\frac{n_t}{n_s}\succ A$ 
    such that any two consecutive intersections are at least  $L_t\gmul
    \frac{L_sn_se^{t-s}}{n_t}$ away. 
    
    Let $\calP$ be an $\ep_B$--short pants decomposition  on $X_t$ and 
    $m=\underset{\beta \in \calP}{\min}\I(\beta,\gamma)$. From 
    \eqnref{Intersection} we have
    \[
      A \lmul \log m. 
    \]
    Note that, for small values of $A$, part (II) follows from Part (I).
    Hence, we assume $A$ is large, which implies in particular that
    $\gamma$ intersects every curve in $\calP$. Even though part (II) seems
    to be more general, this last condition is used in an essential way in
    the proof of Part (II) and the proof does not naturally extend to prove
    part (I). 
    
    We also have 
    \[
    m \lmul \frac{\ell_{X_t}(\gamma)}{\delta_B}  
      \leq e^{t-s} \frac{\ep_B}{\delta_B} \lmul e^{t-s}.
    \]
     
    Cut the segment $\left[\tf(p_l),\tf(p_r)\right]$ into $m \, n_s$ equal
    pieces and let $\tomega$ be one of them.  Denote the projection of
    $\tomega$ to $X_t$ by $\omega$. We would like to show that $\omega$
    intersects a curve in $\calP$. As in the proof of
    \clmref{Intersection}, if $\omega$ is not simple, then it wraps around
    a simple closed $\lambda \in \lambda_G$ and assuming $s_0>
    \log\frac{\epsilon_B}{\delta_B}$, it has to intersect a curve in
    $\calP$. Hence we assume $\omega$ is simple.
    
    Assume for contradiction that $\omega$ is disjoint from $\calP$. Then
    $\omega\subset P$ for some pair of pants $P$ with $\ep_B$-short
    boundaries. It follows from \lemref{Arc-In-Pants} that there is $\beta
    \subset \partial P$ such that $U(\beta)$ contains an endpoint of
    $\omega$ and such that
    \begin{equation} \label{Eqn:OL} 
      \ell_{X_t}(\omega)\ladd
      2(\I(\omega,\tau_\beta)\ell_{X_t}(\beta)+\ell_{X_t}(\tau_\beta)).
    \end{equation}
   
    Then $\gamma$ intersects $\beta$ at least $m$ times. Pick any 
    of the sub-arcs $\sigma$ of  $\gamma$ that connect both boundary components of
    $U(\beta)$. Then, 
    \begin{equation}\label{Eqn:OmegaTauA}
      \I(\omega,\tau_\beta)\ladd \I(\omega,\sigma)+\I(\sigma,\tau_\beta)
      \ladd \frac{1}{m}\I(\omega,\gamma)+\I(\sigma,\tau_\beta).
    \end{equation}
    The last inequality holds because $\omega$ intersects every component of 
    $\gamma^\star \cap U(\beta)$ essentially the same number of times.
    Also, 
    \begin{equation}\label{Eqn:SigmaTauA}
      \I(\sigma,\tau_\beta)\,\ell_t(\beta)+
      \ell_{X_t}(\tau_\beta)\ladd\ell_{X_t}(\sigma)\ladd\frac{1}{m}\ell_{X_t}(\gamma).
    \end{equation}
    and
    \begin{equation} \label{omega-gamma}
     \I(\omega, \gamma) \ladd \frac{\ell_{X_t}(\omega)}{\delta_B \, e^{t-s}}.
    \end{equation}
    From the last four equations and using $\ell_t(\gamma)\leq \epsilon_Be^{t-s}$ and 
    $\ell_{X_t}(\beta) \leq \ep_B$, we have:
    \[
      \ell_{X_t}(\omega)\ladd
     \ell_{X_t}(\omega)\frac{2\epsilon_B}{m\, \delta_B \, e^{t-s}}+ 
     \frac{2\epsilon_B\, e^{t-s}}{m}.
    \]
    But $e^{t-s}\gmul m$. Hence, for some uniform constant $C$    
    \begin{equation}\label{Eqn:OmegalengthA}
      \ell_{X_t}(\omega)\left(1- \frac{2\epsilon_B}{m\delta_Be^{t-s}}
      \right)\leq \frac{Ce^{t-s}}{m}.
    \end{equation}
    The expression in parentheses on the left side is strictly positive since
    we have assumed $s_0>\log\frac{\epsilon_B}{\delta_B}$. 
    Finally, if we choose $L_0$ such that 
    \[ 
      L_0\left(1- \frac{\epsilon_B}{\delta_Be^{s_0}} \right)>C
    \]
    then \eqnref{OmegalengthA} contradicts 
    \[
    \ell_{X_t}(\omega) \geq \frac{L_s \, e^{t-s}}{m}.
    \] 
    Contradiction proves that $\omega$ intersects some curve in $\calP$. 
    
    There are at least $m\, (n_s-4)$ such sub-segments in
    $\left[\tf(p_{l+1}),\tf(p_{r-1})\right]$ and each intersects a lift of
    a curve in $\calP$. If we choose every other segment, we can guarantee
    that the distance along $\lambda$ between these intersection points is
    larger than $L_t = e^{t-s} L_s$. Color these segment according to which
    curve in $\calP$ their projection to $X_t$ intersects and let $\beta$
    be the curve used most often. Then the number $n_t$ of segments
    intersecting a lift of $\beta$ satisfies $n_t \gmul m \, n_s$. Applying
    \propref{Suf-Horizontal} finishes the proof. \qedhere

  \end{proof}
    
\section{Shadow to the curve graph}
  
  \label{Sec:Shadow}

   To show that the shadow of a Thurston geodesic to the curve graph is a
   quasi-geodesic, we construct a retraction from the curve graph to the
   image of the shadow sending a curve $\alpha$ to the shadow of the point
   in the Thurston geodesic where $\alpha$ is \emph{balanced}. 

  \subsection{Balanced time for curves}
  
  Let $n_0$ and $L_0$ be the constants of \thmref{Horizontal}.

  \begin{definition}[Balanced time]

    Let $\GL \colon [a,b] \to \T$ be a Thurston geodesic segment. For any
    curve $\alpha$, let \[ t_\alpha = \inf  \Big\{ t\in [a,b] \ST \text{
    $\alpha$ is $(n_0,L_0)$--horizontal at $\GL(t)$} \Big\} \] Let
    $t_\alpha = b$ if the above set is empty. We refer to $t_\alpha$ as the
    \emph{balanced time} of $\alpha$ along \GL. 

  \end{definition}
  
  Recall the shadow map $\pi \colon \T \to \cc$ from \secref{ShadowMap}.
  The following theorem asserts that the shadow map is a coarse Lipschitz
  map.

  \begin{theorem} \label{Thm:Balanced}
     
    Let $\GL \colon [a,b] \to \T$ be a Thurston geodesic, and let $\pi
    \colon \T \to \cc$ be the shadow map. Suppose
    $\alpha$ and $\beta$ are disjoint curves with $t_\beta \ge t_\alpha$.
    Then $\pi \circ \GL([t_\alpha, t_\beta])$ has uniformly bounded
    diameter in \cc. 

  \end{theorem}
  
  In the following, we develop some notions that will be used to prove
  \thmref{Balanced}. 
  
  Let $X$ be a hyperbolic surface. A \emph{rectangle} $R$ in $X$ is the
  image of a continuous map $\phi \colon [0,a] \times [0,b] \to X$ such
  that $\phi$ is a homeomorphism on the interior of $[0,a] \times [0,b]$
  and the image of each boundary segment of $[0,a] \times [0,b]$ is a
  geodesic arc in $X$.  
  
  \begin{definition} \label{Def:Corridor} 
    
    Let $\gamma$ be a simple closed geodesic on $X$, $\omega$ be a geodesic
    arc and $R$ be a rectangle given by $\phi \colon [0,a] \times [0,b] \to
    X$. We say $R$ is an $(n,L)$--corridor generated by $\gamma$ and
    $\omega$, if \begin{itemize}

      \item Edges $\{0\} \times [0,b] $ and  $\{a\} \times [0,b]$ are
        mapped to sub-arcs of $\omega$.

      \item There are $0 = t_1< \ldots < t_n = b$ such that each $[0,a]
        \times \{t_i\}$ is mapped to a sub-arc of $\gamma$, for all
        $i=1,\ldots,n$.

      \item Arcs $\phi \big( [t_i,t_{i+1}] \times \{0\} \big)$ and $\phi
      \big( [t_i,t_{i+1}] \times \{a\} \big)$ have lengths at least $L$.

    \end{itemize} 
  
  \end{definition}
 

  \begin{figure}[ht]
  \setlength{\unitlength}{0.01\linewidth}
  \begin{picture}(100, 26)
  \put(16.5,-2){ 
  \includegraphics{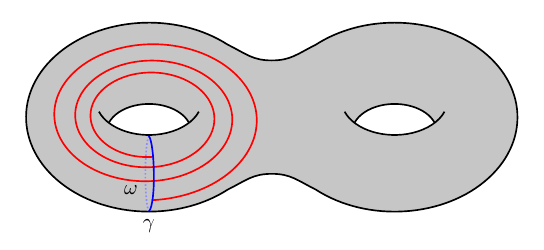}
  }
  \end{picture}
  \caption{A corridor generated by $\gamma$ and $\omega$.}
  \label{Fig:Rectangle}
  \end{figure} 
 
  \begin{lemma} \label{Lem:Corridor}
     
    Let $X$ be a hyperbolic surface and $\gamma$ an $\ep_B$--short curve on
    $X$. For any constants $n$ and $L$, let $\omega$ be a simple geodesic
    arc (possibly closed) with endpoints on $\gamma$ such that \[
      \I(\gamma, \omega) \ge C(n,L) = (6|\chi(X)|+1)\, n \left\lceil
    \frac{L}{\delta_B} \right\rceil +3|\chi(X)|+1. \] Then there exists an
    $(n,L)$--corridor generated by $\gamma$ and $\omega$.

  \end{lemma}
  
  \begin{proof}
     
    Fix $n$ and $L$, and denote by $C=C(n,L)$. Let $\I(\omega, \gamma) = N
    \geq C$. The closure of each connected component of $X \setminus
    \{\gamma \cup \omega\}$ is a surface with a piecewise geodesic
    boundary. Let $Q$ be a complementary component. We refer to points in
    the boundary of $Q$ where two geodesic pieces meet as an angle. Define
    the total combinatorial curvature of $Q$ to be \[ \kappa(Q) = \chi(Q) -
    \frac{\text{\# of angles in $\partial Q$} }4. \] We can represent $X$
    combinatorially with all angles having value $\pi/2$ to obtain \[
    \sum_Q \kappa(Q) = \chi(X). \] Note that, for every component $Q$ that
    is not a rectangle, $\kappa(Q) < -\frac 12$ and hence, the number of
    components that are not rectangles is bounded by $2|\chi(X)|$. In fact,
    the number of angles that appear in non-rectangle components is at most
    $12 |\chi(X)|$, because the ratio of the number of angles to Euler
    characteristic is maximum in the case of a hexagon. Since the total
    number of angles is $4N-4$ (there are only two angles at the first and
    the last intersection points) the number of rectangles is at least: \[
    \frac{(4N - 4) - 12 |\chi(X) | }4 = N - 3 |\chi(X)| -1 \geq
    (6|\chi(X)|+1) \, n \left\lceil \frac{L}{\delta_B} \right\rceil. \]   
    
    We will say two rectangle components can be joined if they share an arc
    of $\gamma$. A \emph{maximal sequence of joined rectangles} is a
    sequence $\{Y_1, \ldots Y_s\}$ of rectangles in $X \setminus \{\omega
    \cup \gamma\}$ such that $Y_i$ and $Y_{i+1}$ can be joined for $i=1,
    \ldots, (s-1)$ and such that $Y_1$ and $Y_s$ share a boundary with a
    non-rectangle component. The number of edges of rectangles that share
    with a non-rectangle component is at most 2 more than the number of
    angles of non-rectangle components (again coming from the first and
    last intersection points of $\gamma$ and $\omega$).  That is, the
    number of maximal sequences of joined rectangles is at most \[ \frac{2
    \cdot (\text{ \# of rectangles})}{12|\chi(X)|+2}. \] Therefore, there
    must be at least one maximal sequence of joined rectangles $\{ Y_1,
    Y_2, \ldots, Y_M \}$ where $M \ge n \lceil L /\delta_B \rceil$. For
    each $i=1, \ldots, M$, the sides of $Y_i$ coming from arcs of $\omega$
    have endpoints on $\gamma$ and thus are at least $\delta_B$ long.
    Therefore the union $\bigcup_{j=1}^M Y_i$ is an $(n,L)$-corridor for
    $\gamma$ after letting $t_i$ be the point that maps to the intersection
    number $ i \cdot \lceil L /\delta_B \rceil$. \qedhere
     
  \end{proof}
  
  \begin{proposition} \label{Prop:Balanced}

    Let $n_0$ and $L_0$ be the constants from \thmref{Horizontal}. There
    exists a constant $n_1$ such that for any $n \ge n_1$ and $L \ge L_0$,
    if $\alpha$ is $(n,L)$--horizontal at $\GL(t)=X_t$, then any curve
    $\beta$ disjoint from $\alpha$ is either $(n_0,L_0)$--horizontal at
    $X_t$ or $\dS(X_t, \beta) = O(1)$. 
     
  \end{proposition}
  
  \begin{proof}
     
    Let $n_1 = 3C(n_0,L_0)$ (see \lemref{Corridor}). Also let
    $n \ge n_1$ and $L \ge L_0$.
    
    Suppose $\alpha$ is $(n,L)$--horizontal at $X_t$. Let $\gamma$ be an
    $\ep_B$--short curve on $X_t$,  $\tlambda$ be a lift of a leaf of
    $\lambda_\GL$, $\talpha$ be a lift of $\alpha$ and $\{\tgamma_1,
    \ldots, \tgamma_n\}$ be $n$--lifts of $\tgamma$ together satisfying
    \defref{Horizontal}. Choose a most central segment $\tomega \subset
    \talpha$ between $\tgamma_1$ and $\tgamma_n$ such that $\tomega$
    intersects $C(n_0,L_0)$ lifts of $\gamma$, including two intersections
    coming from the endpoints of $\tomega$. Let $\omega$ be the projection
    of $\tomega$ to $X_t$. If $\I(\omega, \gamma) < C(n_0,L_0)$, then
    $\omega = \alpha$ and we are done since \[ \dS(X_t, \beta) \ladd
    \dS(\gamma, \alpha) \lmul \log C(n_0,L_0) = O(1). \] Otherwise, by
    \lemref{Corridor}, there exists $(n_0,L_0)$--corridor $R$ generated by
    $\gamma$ and $\omega$. Let $\phi \colon [0,a] \times [0,b] \to X_t$ be
    the map whose image is $R$ satisfying the conditions of
    \defref{Corridor}. 
       
    Lift $\phi$ to the map \[ \tilde{\phi} \colon [0,a] \times [0,b] \to
    \tx_t \qquad\text{with}\qquad \tilde{\phi} \big( \{0\} \times [0,b]
    \big) = \tomega. \] Note that $\tilde{\phi} \big( \{a\} \times [0,b]
    \big)$ is a translate of a sub-arc $\tomega'$ of $\talpha$ by an
    isometry of $\HH^2$ fixing $\tlambda$. Also, $\tomega'$ intersects the
    same number of $\tgamma_i$ and $\tomega'$ and $\tomega$ intersect. But
    $n \geq 3C(n_0,L_0)$ and $\tomega$ was central. Thus $\tomega'$ is
    still between $\tgamma_1$ and $\tgamma_n$ and hence is $\ep_B$--close
    to $\tlambda$. 

    Since $\beta$ is disjoint from $\alpha$, if $\beta$ intersects $R$ it
    has to enter from the edge $\phi([0,a] \times \{0 \})$, travel through
    the corridor and exit from the edge $\phi([0,a] \times \{b\})$.
    Therefore, there must exist a lift $\tbeta$ of $\beta$ passing through
    $\tilde{\phi} \big( (0,a) \times [0,b] \big)$ intersecting every
    $\tilde{\phi} \big( [0,a] \times \{t_i\} \big)$. That is, $\tbeta$
    intersects $n_0$ of $\gamma_i$ at a distance at most $\ep_B$ from
    $\tlambda$. Thus, by \propref{Suf-Horizontal} $\beta$ is
    $(n_0,L_0)$--horizontal at $X_t$ if $\beta$ intersects $R$. 

    Now suppose $\beta$ is disjoint from $R$. Fix a parametrization $\psi
    \colon [0,c] \to X_t$ for $\omega$. Let $\omega_1 = \phi \big( \{0\}
    \times [0,b] \big)$ and $\omega_2 = \phi \big( \{a\} \times [0,b]
    \big)$. The parametrization $\psi$ traverses $\omega_1$ or $\omega_2$
    either in same or opposite direction as $\phi$ and either traverses
    $\omega_1$ before $\omega_2$ or vice versa. We assume $\psi$ traverses
    $\omega_1$ in the same direction and speed as $\phi$ and $\psi$
    traverses $\omega_1$ before $\omega_2$ (the proofs in the other cases
    are similar). Let \[ 0 \le s < t < s+b < t+b \le c \] be such that
    $\omega_1= \psi([s,s+b])$ and $\omega_2= \psi([t, t+b])$. Let $\omega'
    = \psi([s,t])$ and $\gamma' = \phi \big( [0,a] \times \{0\} \big)$.
    Consider the curve $\eta$ that is a concatenation of $\omega'$ and
    $\gamma'$. Topologically, $\eta$ is a non-trivial simple closed curve
    with $\I(\eta, \gamma) \le \I(\omega, \gamma) = C(n_0,L_0)$. Since
    $\beta$ is disjoint from $R$, it is disjoint from $\eta$. Therefore, \[
    \dS ( X_t, \beta) \ladd \dS( \gamma, \eta) = O(1). \] This concludes
    the proof of the proposition. \qedhere
     
  \end{proof}

  \begin{proof}[Proof of \thmref{Balanced}]
    
    The proof now follows from \propref{Balanced} and \thmref{Horizontal}.

    Let $n_1$ be as in \propref{Balanced} and $L_0$ be as in
    \thmref{Horizontal}. Let $s > t_\alpha$ be the first time in $[a,b]$
    that $\alpha$ is $(n_1,L_0)$--horizontal at $s$ (let $s=b$ if this
    never happens). By \thmref{Horizontal}, $\pi \circ \GL([t_\alpha, s])$
    has uniformly bounded diameter in \cc. If $s \ge t_\beta$, then we are
    done. Otherwise, $s < t_\beta$ and for any $t \in [s, t_\beta)$, by
    \propref{Balanced}, $\dS(X_t, \beta) = O(1)$. Therefore, $\pi \circ
    \GL([s,t_\beta])$ also has uniformly bounded diameter in \cc. \qedhere

  \end{proof}
  
  \subsection{Retraction}
  
  \begin{theorem} \label{Thm:Retraction}

    Given a Thurston geodesic $\GL \colon [a,b] \to \T$, the map $\cc \to
    \pi \circ \GL \big( [a,b] \big) \subset \cc$ taking a curve $\alpha$ to
    $\pi(X_{t_\alpha})$ is a coarse Lipschitz retraction.

  \end{theorem}
  
  Before proving \thmref{Retraction}, we show how to derive
  \thmref{Shadow}. First we give a precise definition of
  \emph{reparametrized quasi-geodesic}. 
 
  Fix a constant $K > 0$. We will call a path $\phi \colon [a,b]  \to
  \calX$ in a metric space $\calX$ a $K$--quasi-geodesic if for all $a\le s
  \le t \le b$, \[ \frac{1}{K}(t-s) -K \le d_{\calX} \big( \phi(s),\phi(t)
  \big) \le K(t-s) + K.\] We will say $\phi$ is a \emph{reparametrized}
  $K$--quasi-geodesic if there is an increasing function $h \colon [0,n]
  \to [a,b]$ such that $\phi \circ h$ is a $K$--quasi-geodesic.
  Furthermore, for all $i \in [0,n-1]$, we have $\diam_{\calX} \big(
  [\phi(h(i),\phi(h(i+1))] \big) \le K$. In the case that $h$ is not onto,
  we also require that $\diam_{\calX} \big( [\phi(a),\phi(h(0))] \big) \le
  K$ and $\diam_{\calX} \big( [\phi(h(n)),\phi(b)] \big) \le K$. A
  collection $\{ \phi_i \}_{i \in I}$ of reparametrized quasi-geodesics is
  uniform if there is a constant $K$ that works for the collection. 
  
  The following is a restatement of \thmref{Shadow}.

  \begin{theorem} \label{Cor:QG}

    The collection of $ \{ \pi \circ \GL \colon [a,b] \to \cc \}$ ranging
    over Thurston geodesics $\GL \colon [a,b] \to \T$ is a uniform family
    of reparametrized quasi-geodesics in \cc. 

  \end{theorem}
  
  \begin{proof}
    
    This argument is standard \cite{Bow06} and follows easily from
    \thmref{Retraction}. 
    
    Let $\GL \colon [a,b] \to \T$ be a Thurston geodesic. Let $\alpha \in
    \pi\circ \GL(a)$ and $\alpha' \in \pi \circ \GL(b)$ be two curves.
    Choose a geodesic $\alpha=\alpha_0, \ldots, \alpha_n = \alpha'$ in \cc.
    Let $t_i$ be the balanced time of $\alpha_i$ along \GL, and let
    $t_{n+1}=b$. By \thmref{Retraction}, $\diamS \big( [\GL(t_i),
    \GL(t_{i+1})] \big)$ is uniformly bounded. The $t_i$'s may not occur
    monotonically along $[a,b]$, but for each $0\leq i\leq n$, there exists
    $j \geq i$, such that $t_j \le t_i \le t_{j+1}$, and $\diamS \big(
    [\GL(t_i), \GL(t_{j+1})] \big) \le \diamS \big( [\GL(t_j),
    \GL(t_{j+1})] \big) = O(1)$. Thus, there is a sequence $0=i_0 < i_1 <
    \cdots < i_k=n$, such that $t_{i_{j+1}} > t_{i_j}$ and $\diamS \big(
    [\GL(t_{i_j}), \GL( t_{i_{j+1}}) ] \big)$ is uniformly bounded. We will
    call such a sequence admissible and choose one with minimal length $k$.
    For simplicity, we will relabel each $t_{i_j}$ by $t_j$. 
    
    Now let $h \colon [0,k] \to [a,b]$ be defined by sending each
    subinterval $[j,j+1]$ to $[t_j,t_{j+1}]$ by a linear map, for all
    $j=0,\ldots, k-1$. By \thmref{Retraction}, $\diamS \big(
    [\GL(a),\GL(t_0)] \big) = O(1)$ and $\diamS \big([\GL(t_k),\GL(b)]
    \big) = O(1)$. Set $\calG_i = \calG \circ h(i)$. By construction,
    $\diamS \big( [\GL_i, \GL_{i+1} ]\big) = O(1)$, for all $i \in [0,k-1]$
    and $k \le n = \dS(\alpha,\beta)$. Therefore, for all $ 0 \le i \le i'
    \le k$, we have \[ \dS \big( \GL_i, \GL_{i'} \big) \prec i'-i. \]
    
    The only thing remaining to check is that the lower bound for the
    definition of quasi-geodesic: that is, for all $0 \le i \le i' \le k$,
    we will show \[ i'-i \le \dS \big( \GL_i,\GL_{i'} \big) + 2. \] It is
    enough to prove this for $i, i' \in \{0 , \ldots k\}$. For $0 \le i <
    i' \le k$, let $\beta \in \pi \circ \GL_i = \pi \circ \GL(t_i)$ and
    $\beta' \in \pi \circ \GL_{i'} =  \pi \circ \GL(t_{i'})$. Let $m =
    \dS(\beta,\beta')$ and choose a geodesic $\beta=\beta_0, \ldots,
    \beta_m = \beta'$ in \cc. Let $s_i$ be the balance time of $\beta_i$
    along \GL. After choosing a subsequence, we may assume $s_i$ appear
    monotonically along $[a,b]$ and that $t_i < s_0$ and $s_m < t_{i'}$. We
    can modify the admissible sequence \[ t_0 < \cdots < t_i < \cdots <
    t_{i'} < \cdots < t_k \] by
    \[ t_0 < \cdots < t_i < s_0 < \cdots < s_m < t_{i'} < \cdots < t_k, \]
    which is still admissible. By minimality of $k$, we must have $i'-i \le
    m+2$ which proves what we want. \qedhere
    
  \end{proof}
    
    The proof of \thmref{Retraction} requires some technical results about
    hyperbolic surfaces.

    Given a hyperbolic surface $X$, consider a simple closed geodesic
    $\gamma$ that is $\ep_B$--short and a simple geodesic $\lambda$ on $X$.
    Assuming that $\lambda$ intersects $\gamma$ many times, we would like
    to find a simple closed curve $\alpha$ with a uniformly bounded
    intersection number with $\gamma$ that is $(n_0, L_0)$--horizontal. The
    argument here is somewhat delicate since there are essentially two
    possible situations; either $\lambda$ twists around a relatively short
    curve $\alpha$ or $\alpha$ is somewhat longer and a long subsegment of
    it stays close to $\lambda$ in the universal cover. 

    We find the appropriate curve $\alpha$ by applying surgery between
    $\lambda$ and $\gamma$ such that $\alpha$ contains a long sub-segment of
    $\lambda$. But we also need to have some control such that after pulling
    $\alpha$ tight it still stays close to $\lambda$. The following two
    lemmas will give the needed control. 
    
    In the following, orient the curve $\gamma$ so the left side and the
    right side of $\gamma$ are defined. We will say an arc $\omega$ with
    endpoints on $\gamma$ hits $\gamma$ on opposite sides if the two
    endpoints of $\omega$ are on different sides of $\gamma$; otherwise,
    $\omega$ hits $\gamma$ on the same side. 

    \begin{lemma} \label{Lem:QuasiCurve}
     
      Let $\gamma$ and $\lambda$ be as above and let $\alpha = \eta \cup
      \omega$ be a closed curve in $X$ that is obtained from concatenation
      of a sub-arc $\eta$ of $\gamma$ and a sub-arc $\omega$ of $\lambda$.
      Also, assume that $\eta$ hits $\omega$ on opposite sides and
      $L=\ell_X(\omega) \geq 4 \ep_B$. Then, $\alpha^*$, the geodesic
      representative of $\alpha$ in $X$, stays in an
      $O(\ep_B)$--neighborhood of $\alpha$.
        
  \end{lemma}
  
  \begin{proof}
    
    This is a well known fact in hyperbolic geometry. We sketch the proof
    here. Consider the lift of $\alpha$ to $\HH^2$ as a concatenation of
    segments $\etab_i$ and $\bomega_i$ that are lifts of $\eta$ and
    $\omega$ respectively. Since $\omega$ is a sub-arc of a complete simple
    geodesic in $X$, the segments $\bomega_i$ lie on complete geodesics
    $\tomega_i$ in $\HH^2$ that are disjoint. The condition that $\eta$
    hits $\omega$ on opposite sides means that $\bomega_{i+1}$ does not
    backtrack along $\omega_i$.
  
    Fixing $o$, the center of segment $\omega_0$, as the center of the
    Poincar\'e disk, the Euclidean distance between and endpoints of
    geodesics $\tomega_i$ in $\partial \HH^2$ form a Cauchy sequence (in
    fact, they decrease exponentially fast) and hence they converge.
    Namely, the visual angle at $o$ of the endpoint of $\tomega_1$ is at
    most $O(\ep_Be^{-L/2})$ and the visual angle at $o$ between the
    endpoints of $\tomega_i$ and $\tomega_{i+1}$ decrease exponentially
    with $|i|$. Hence the lift of $\gamma^*$ starts and ends near the
    endpoints of $\tomega_0$ with a visual angle of $O(\ep_B e^{-L/2})$.
    Therefore, an $O(\ep_B)$--neighborhood of lift of $\gamma^*$ contains
    $\bomega_0$. This finishes the proof. \qedhere 
    
  \end{proof} 
  
  \begin{lemma} \label{Lem:QuasiCurve2}
     
    Let $\beta$ and $\beta'$ be simple closed curves in $X$ (possibly
    $\beta=\beta'$) with lengths longer than $\delta_B$ and let $\eta$ be a
    geodesic segment that is disjoint from both. Let $\gamma$ and $\gamma'$
    be two segments of length $O(1)$ connecting the endpoints of $\eta$ to
    $\beta$ and $\beta'$ respectively so the curve $\alpha$ obtained by
    the concatenation
    \[
    \beta \cup \gamma \cup \eta \cup \gamma' \cup \beta'
    \cup \gamma' \cup \eta \cup \gamma
    \]
    is simple. Let $\alpha^*$ be the geodesic representative of $\alpha$.
    Then, in the universal cover,  any lift of $\eta$ is contained in a
    bounded neighborhood of the union of a lift of $\alpha^*$, a lift of
    $\beta$ and a lift of $\beta'$. 
            
  \end{lemma}

  \begin{proof}
  
  The lemma is non-trivial because two copies of $\eta$ is used and they
  may backtrack each other. It is essential that there is a lower bound on
  the lengths of $\beta$ and $\beta'$ and the lemma essentially follows
  from \lemref{Diverge}.  
    
  Consider a lift of $\alpha$ to the universal cover. Ignoring the lifts of
  $\gamma$ and $\gamma'$ which have bounded length, we consider the
  segments $\beta_i \cup \eta_i \cup \beta_i' \cup \overline{\eta}_i$ where
  $\eta_i$ and $\overline{\eta}_i$ are lifts of $\eta$, $\beta_i$ are lifts
  of $\beta$, $\beta_i'$ are lifts of $\beta'$ and endpoints of every
  segment are in a uniformly bounded neighborhood of an endpoint of the
  next segment.
    
  The segments $\eta_i$ and $\overline{\eta}_i$ lie on geodesics
  $\lambda_i$ and $\blambda_i$ that are disjoint. In fact, there is an
  isometry of $\HH^2$, associated to the curve $\beta'$, whose axis
  contains $\beta_i'$ and sends $\eta_i$ to $\overline{\eta}_i$. Let
  $\delta_0$ be large enough such that the $\delta_0$--neighborhood of
  $\beta'$ contains the standard collar $U(\beta')$. Then $\delta_0$ is a
  universal constant since there is a lower bound on the length of
  $\beta'$. Let $U_{\delta_0}(\beta')$ be the $\delta_0$--neighborhood of
  $\beta'$ and $U_i$ be the lift that contains $\beta_i'$. There is a
  universal lower bound on the length of the boundary of
  $U_{\delta_0}(\beta')$ which means there is a lower bound on the distance
  between the intersection points of $\eta_i$ and $\overline{\eta}_i$ with
  $U_i$. It now follows from \lemref{Diverge} that there is a lower bound
  of $d_0$ for the distance between the subsegments of $\eta_i$ and
  $\overline{\eta}_i$ that are outside of $U_i$. That is, if $\eta_i$ is
  not near $\beta_i$, it cannot be too close to $\overline{\eta}_i$ and
  hence $\eta_i$ and $\overline{\eta}_i$ do not fellow travel for a long
  time outside of a uniform neighborhood of $\beta_i'$. A similar statement
  is true for $\overline{\eta}_i$, $\eta_{i+1}$ and $\beta_i$. 
  
  Since $\HH$ is Gromov hyperbolic, the lift of $\alpha^*$ is contained in
  a uniform neighborhood of segments  $\beta_i \cup \eta_i \cup \beta_i'
  \cup \overline{\eta}_i$. In fact, each point in any of these segments is
  either close to the lift of $\alpha^*$ or close to a point in some other
  segment. We have shown that $\eta_i$ and $\overline{\eta}_i$ do not
  fellow travel for large subsegment. This means any point in $\eta_i$ is
  close to either $\beta_i$, $\beta_i'$, or to the lift of $\alpha^*$.
  \qedhere
  
  \end{proof}

  We would like to show that, at every time $t$, there is a curve $\alpha$
  which has balanced time $t_\alpha=t$ and whose distance in \cc from the
  shadow of $X_t$ is uniformly bounded.  The next proposition shows that a
  coarse version of this statement holds. 

  \begin{proposition} \label{Prop:HorExist}
     
    Let $\GL \colon [a,b] \to \T$ be a Thurston geodesic segment. For any
    $t \in [a,b]$, if $\lambda_{\GL}$ intersects some short curve on
    $X_t=\GL(t)$, then there exists an $(n_0,L_0)$--horizontal curve
    $\alpha$ on $X_t$ such that $\I(\alpha,\gamma) = O(1)$ for any
    $\ep_B$--short curve $\gamma$. Furthermore, $\dS(X_{t_\alpha},X_t) =
    O(1)$.
     
  \end{proposition}

  \begin{proof}
   
    First, we will construct $\alpha$. Let 
    \[
      N_0= n_0\, \left\lceil \frac{L_0}{\delta_B} \right\rceil + K,
    \] 
    where $K$ is the additive error coming from \lemref{QuasiCurve2}.

    If $\lambda_{\GL}$ has a closed leaf $\lambda$ that intersects every
    short curve on $X_t$ at most $5N_0$--times, then $\alpha=\lambda$ has
    the desired properties. Otherwise, we can fix a leaf $\lambda$ in the
    stump of $\lambda_\GL$ that intersects some $\ep_B$--short curve more
    than $5N_0$ times. 

    Fix a segment $\omega$ of $\lambda$ such that $\omega$ has endpoints on
    a short curve $\gamma$ with $\I(\gamma, \omega) = 5N_0$, and $\omega$
    intersects all other short curves at most $5N_0$ times. Orient
    $\gamma$ such that we can talk about the two sides of $\gamma$. We will show
    that, applying a surgery between $\omega$ and $\gamma$, we can obtain a
    simple closed curve $\alpha$ that still intersects $\gamma$ and stays
    close to $\lambda$ for a long time. Unfortunately, this process is
    delicate and depending on the intersection pattern of $\gamma$ and
    $\omega$ we may have to apply a different surgery.  
    The conclusion will follow if either of the following two cases occur.
    
    {\bf Case (1)}: There is a sub-arc $\eta$ of $\omega$ that hits
    $\gamma$ on opposite sides with $\I(\eta,\gamma) \ge N_0$, and the
    endpoints of $\eta$ can be joined by a segment of $\gamma$ that is
    disjoint from the interior of $\eta$. 

    In this case, the geodesic representative $\alpha$ of the concatenation
    of $\eta$ and a segment of $\gamma$ is $(n_0,L_0)$--horizontal by
    \lemref{QuasiCurve} (see left side of \figref{Surgery1}). 
       
    {\bf Case (2)}: There is a sub-arc $\eta$ of $\omega$ and a closed
    curve $\beta$ disjoint from $\eta$, such that $\I(\eta, \gamma) \ge
    N_0$ and $\ell_X(\beta) \ge \delta_B$, and the endpoints of $\eta$ are
    close to the same point on $\beta$. Furthermore, each endpoint of
    $\eta$ can be joined to a nearby point on $\beta$ by a segment of
    $\gamma$ that is disjoint from $\beta$ and the interior of $\eta$.
      
    In this case (see right side of \figref{Surgery1}), let $\alpha$ be the
    curve obtained by closing up $\eta$ with $\beta$ and one or two sub
    arcs of $\gamma$. If $\eta$ twists around $\beta$ $N_0$--times, then
    $\beta$ is $(n_0,L_0)$--horizontal by \propref{Suf-Horizontal};
    otherwise, by \lemref{QuasiCurve2}, $\alpha$ has a segment that
    intersects $\gamma$ $N_0$--times and stays close to $\eta$, in which
    case $\alpha$ is $(n_0,L_0)$--horizontal by \propref{Suf-Horizontal}.

    {\bf Case (3)}: There is a sub-arc $\eta$ of $\omega$ and there are two
    closed curves $\beta$ and $\beta'$ that are disjoint from $\eta$, such
    that $\I(\eta, \gamma) \ge N_0$, $\ell_X(\beta) \ge \delta_B$ and
    $\ell_X(\beta') \ge \delta_B$, and the two endpoints of $\eta$ are
    close to $\beta$ and $\beta'$. Furthermore, there exists a segment of
    $\gamma$ joining one endpoint of $\eta$ to $\beta$ and a segment of
    $\gamma$ joining the other endpoint of $\eta$ to $\beta'$, such that
    both segments are disjoint from $\beta$, $\beta'$ and the interior of
    $\eta$.

    In this case (see \figref{Surgery2}), let $\alpha$ be the curve
    obtained by gluing two copies of $\eta$, $\beta$, $\beta'$ and a few
    sub-arcs of $\gamma$. If $\eta$ twists around either $\beta$ or
    $\beta'$ $N_0$--times, then either $\beta$ or $\beta'$ is
    $(n_0,L_0)$--horizontal by \propref{Suf-Horizontal}; otherwise, by
    \lemref{QuasiCurve2}, $\alpha$ has a segment that intersects $\gamma$
    $N_0$--times and stays close to $\eta$, in which case $\alpha$ is
    $(n_0,L_0)$--horizontal by \propref{Suf-Horizontal}.

    We now show that at least one of these three cases happens.
    
    Let $p_0$ and $q_0$ be the endpoints of $\omega$. Let $p_{-1}$ and
    $p_1$ be the adjacent intersection points along $\gamma$ to $p_0$ and
    $q_{-1}$ and $q_1$ be the adjacent intersection points along $\gamma$
    to $q_0$. By relabeling if necessary, we may assume $\omega$ passes
    from $p_0$ to $p_1$ and then to $p_{-1}$. Assume $\omega$ passes from
    $q_0$ to $q_1$ and then to $q_{-1}$. We allow the possibility $p_{-1} =
    q_0$ and $q_{-1} = p_0$, or $p_{-1} = q_{-1}$.
   
    \begin{figure}[htp!]
    \setlength{\unitlength}{0.01\linewidth}
    \begin{picture}(100, 23)
    \put(22,-2){
    \includegraphics{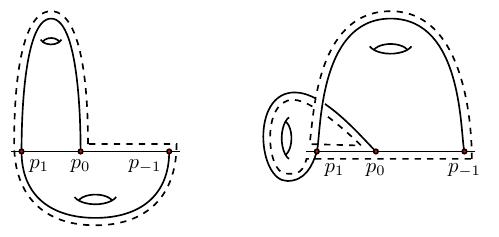}
    }
    \end{picture}
    \caption{Case (1) and (2) of \propref{HorExist}.}
    \label{Fig:Surgery1} 
    \end{figure} 

    \begin{figure} [htp!]
    \setlength{\unitlength}{0.01\linewidth}
    \begin{picture}(100, 18)
    \put(30,-2){ 
    \includegraphics{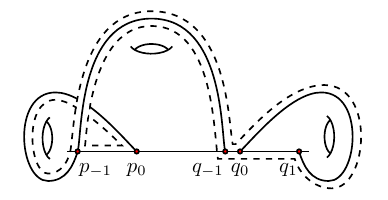} 
    }
    \end{picture} 
    \caption{Case (3) of \propref{HorExist}.}
    \label{Fig:Surgery2} 
    \end{figure} 

    Let $\omega_0$ be the sub-arc of $\omega$ from $p_0$ to $p_1$ and let
    $\omega_1$ be the sub-arc from $p_1$ to $p_{-1}$. Suppose $\omega_0
    \cup \omega_1$ intersects $\gamma$ at least $2N_0$ times. In this
    situation, we have several possibilities that will yield case (1) or
    (2). If $\omega_0 \cup \omega_1$ hits $\gamma$ on opposite sides, as in
    the left side of \figref{Surgery1}, then we are in case (1) with $\eta
    = \omega_0 \cup \omega_1$. Otherwise, one of $\omega_i$ hits $\gamma$
    on opposite sides and the other one hits $\gamma$ on the same side.
    Assume $\omega_0$ hits $\gamma$ on opposite sides, as in the right side
    of \figref{Surgery1}. We are again in case (1) if $\omega_0$ intersects
    $\gamma$ at least $N_0$ times. If not, let $\beta$ be the closed curve
    obtained from closing up $\omega_0$ with an arc of $\gamma$. Since
    $\omega_0$ has endpoints on $\gamma$ and $\beta$ stays close to
    $\omega_0$ by \lemref{QuasiCurve}, $\ell_X(\beta) \ge \delta_B$. We are
    now in case (2) with $\eta = \omega_1$. The dotted line in the right
    side of \figref{Surgery1} represents the closed curve obtained from
    this surgery.
    
    Similarly, let $\omega_0'$ be the sub-arc of $\omega$ from $q_0$ to
    $q_1$ and let $\omega_1'$ be the arc from $q_1$ to $q_{-1}$. As above,
    we are done if $\omega_0' \cup \omega_1'$ intersects $\gamma$ at least
    $2N_0$ times. 

    Since $\omega$ intersects $\gamma$ $5N_0$--times, if neither $\omega_0
    \cup \omega_1$ or $\omega_0' \cup \omega_1'$ intersects $\gamma$ at
    least $2N_0$ times, then the arc $\eta$ from $p_{-1}$ to $q_{-1}$ must
    have at least $N_0$ intersections with $\gamma$. If $\eta$ hits
    $\gamma$ on opposite sides, then we are in case (1). Otherwise, at
    least one of $\omega_0$, $\omega_1$, or $\omega_0 \cup \omega_1$ hits
    $\gamma$ on opposite sides. Close this arc to obtain a closed curve
    $\beta$ which has length at least $\delta_B$. Similarly, let $\beta'$
    be a closed curve obtained from closing up either $\omega_0'$,
    $\omega_1'$, or $\omega_0' \cup \omega_1'$. We are now in case (3). In
    \figref{Surgery2}, we've illustrated the situation when $\omega_0 \cup
    \omega_1$ forms $\beta$ and $\omega_0'$ forms $\beta'$. 

    It remains to show $\dS(X_{t_\alpha}, X_t) = O(1)$. By definition,
    $t_\alpha \le t$ and $\alpha$ is $(n_0,L_0)$--horizontal on
    $X_{t_\alpha}$. Assume $t-t_\alpha \ge s_0$, where $s_0$ is the
    constant of \thmref{Horizontal}. Let $\gamma_\alpha$ be an anchor curve
    for $\alpha$ at time $t_\alpha$. The assumption implies
    $\dS(X_t,\alpha) = O(1)$, so it is enough to show
    $\dS(\gamma_\alpha,\alpha) = O(1)$. Denote
    $D=\dS(\gamma_\alpha,\alpha)$. Recall that $D\leq
    \log_2\I(\alpha,\gamma_\alpha)+1$.
        
     Let $\tlambda$ and $\talpha$ be  as in  \defref{Horizontal}, let
     $\omega$  and $\tau$ be the segments of $\tlambda$ and $\talpha$ which
     are at most $\epsilon_B$ Hausdorff distance apart and which intersect
     $n_0$ lifts of $\gamma_\alpha$. We may assume
     $\I(\alpha,\gamma_\alpha)$ is large enough such that $\tau$ projects to
     a proper sub-arc of $\alpha$, that is, the length of $\tau$ is smaller
     than the length of $\alpha$. Let $\tf$ be the lift of an optimal map
     $f \colon X_{t_\alpha} \to X_t$. By the proof of \thmref{Horizontal},
     up to a multiplicative error $\tf(\omega)$ intersects $n_02^D$ lifts
     of a short curve $\gamma'$ on $X_t$. Moreover, a segment $\tau'$ of
     the geodesic representative of $\tf(\talpha)$  is $\epsilon_B$--close
     to $\tf(\tlambda)$  and  also intersects the $n_02^D$ lifts of
     $\gamma'$ up to a bounded error. The length of $\alpha$ on $X_t$ is
     bigger than the length of $\tau'$, hence $\I(\alpha,\gamma')\gmul
     n_02^D$. But $\I(\alpha,\gamma')=O(1)$ by assumption, therefore $D$
     must be bounded. \qedhere 

  \end{proof}

  \begin{proof}[Proof of \thmref{Retraction}.]
     
    By \thmref{Balanced}, $\pi$ is a coarse Lipschitz map. Let $\alpha$ be
    any short curve on $X_t$ and $\lambda_G$ be the maximally stretched
    lamination. We will show $\diamS \big( [X_t,X_{t_\alpha}]  \big) =
    O(1)$. 

    If (the stump of) $\lambda_G$ is a short curve it may not intersect any
    other short curve at $X_t$. Let $s$ be the first time $\lambda_G$
    intersects some short curve $\gamma$ in $X_s$. Since $\lambda_G$ is a
    short curve in the interval $[t,s]$ we have $\diamS \big( [X_t, X_s]
    \big) \eadd \dS(\alpha,\gamma) = O(1)$.

    We now show $\diamS \big( [X_s, X_{t_\alpha}] \big) = O(1)$. Since
    $\lambda_\GL$ intersects a short curve $\gamma$ on $X_s$, by
    \propref{HorExist}, there exists a curve $\beta$ on $X_s$ with
    $\I(\beta,\gamma) = O(1)$ and $\diamS \big( [X_{t_\beta}, X_s] \big) =
    O(1)$. We have $\dS(\alpha, \beta) \le \dS(\alpha,\gamma) +
    \dS(\gamma,\beta) = O(1)$, which implies by \thmref{Balanced} that
    $\diamS\big( [X_{t_\alpha}, X_{t_\beta}] \big) = O(1)$. The conclusion
    follows by the triangle inequality. \qedhere

  \end{proof}
  
\section{Examples of geodesics} 

  \label{Sec:Examples}
 
  In this section, we construct several examples of geodesics in the
  Thurston metric demonstrating various possible behaviors, proving
  \thmref{Bad-Geodesic} and \thmref{Not-Short} from the introduction. The
  main idea in all these examples is that it is possible for the maximally
  stretched lamination associated to some Thurston geodesic to be contained
  in some subsurface $W$ where the lengths of all curves disjoint from $W$
  (including $\partial W$) stay constant along the geodesic. This contrasts
  the behavior of a \Teich geodesic, where in a similar situation the
  length of $\partial W$ would get short along the geodesic \cite{Raf05}.
  Our construction can be made to be very general. However, in the interest
  of simplicity, we make an explicit construction when $W$ is a torus with
  one boundary component. 
 
  For a constant $\ell$, let $\calT(S_{1,1}, \ell)$ be the space of
  hyperbolic structures on a torus with one geodesic boundary where the
  length of the boundary curve is $\ell$. Note that we can equip
  $\calT(S_{1,1}, \ell)$ with the Thurston metric as usual (see \cite{GK13}
  for details). 

  Let $\mu$ be any irrational measured lamination on \torus. There is a
  unique way to complete $\mu$ to a \emph{complete} lamination $\lambda$,
  such that the complement are two ideal triangles, by adding two bi-infinite
  leaves both tending to the cusp in one direction and wrapping around
  $\mu$ in the other. Hence, for any $U_0 \in \calT(\torus)$, there exists
  a unique stretch path from $U_0$ with $\mu$ the stump of the maximally
  stretched lamination. We will denote by $\str(U_0,\lambda,t)$ this
  stretch path. 

  \begin{proposition} \label{Prop:Fixed-Boundary}

    There is a constant $\ell_0$ such that the following holds. For any
    $U_0 \in \calT(\torus)$, any irrational measured lamination $\mu$ on
    \torus, and the associated stretch path\[ U_t = \str(U_0, \lambda, t),
    \qquad U_t \in \calT(\torus), \] where $\lambda$ is the unique complete
    lamination containing $\mu$ as its stump, there is a bi-infinite Thurston
    geodesic $W_t \in \calT(\torus, \ell_0)$, $t\in \RR$, where the stump
    of the maximally stretched lamination is still $\mu$ and, for any other
    curve $\alpha$ in \torus, 
    \[
      \ell_{W_t}(\alpha) \emul \ell_{U_t}(\alpha). 
    \] 

  \end{proposition}
  
  \begin{proof}
    
    We will refer to \figref{Double} for this proof. 
    
    Choose $U_0 \in \calT(\torus)$ and represent $\mu$ as a geodesic
    lamination on $U_0$. Let $\lambda$ be the unique completion of $\mu$
    and let $A$ and $B$ be the ideal triangles in the complement of
    $\lambda$ in $U_0$.  Each $A$ and $B$ has two sides $a$ and $b$ coming
    from the two leaves of $\lambda$ tending toward the cusp, and a third
    side that wraps around the stump $\mu$. There is an involution of $U_0$
    fixing $\lambda$ and switching $A$ with $B$. Hence, the two anchors
    points of $A$ at $a$ and $b$ are glued respectively to the two anchor
    points of $B$ at $a$ and $b$. See the upper right-hand side of
    \figref{Double}.
    
    \begin{figure}[htp!]
    \setlength{\unitlength}{0.01\linewidth}
    \begin{picture}(100, 58)
    \put(3,-2){ 
    \includegraphics[totalheight=0.45\textheight]{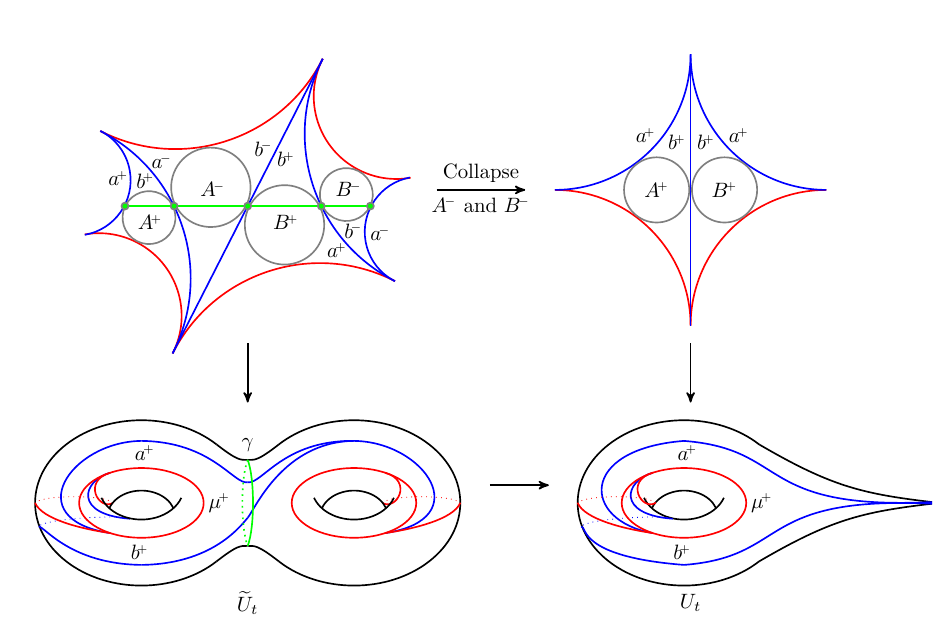}
    }
    \end{picture}
    \caption{We can double the stretch path on a punctured torus 
    to obtain a stretch path in a surface of genus two. The length of
    the curve $\gamma$ remains unchanged along the stretch path. }
    \label{Fig:Double}
    \end{figure}

    Now we double this picture. Let $U_0^+=U_0$ and let $U_0^-$ be an
    orientation reversing copy of $U_0^+$. We also label $A=A^+$, $B=B^+$,
    $a=a^+$, $b=b^+$ and $\mu=\mu^+$ and we label the associated objects in
    $U_0^-$ by $A^-$, $B^-$, $a^-$, $b^-$ and $\mu^-$. Cut $U_0^\pm$ open
    along $a^\pm$ and $b^\pm$. Via a reflection map, glue $b^+$ of $A^+$ to
    $a^-$ of $A^-$ such that the anchor point of $A^+$ in $b^+$ is glued to
    the anchor point of $A^-$ in $a^-$. Similarly, via a reflection map,
    glue $b^+$ to $b^-$, $a^+$ to $b^-$ and $a^+$ to $a^-$ gluing the
    corresponding anchor points. Note that the third sides of $A^{\pm}$ and
    $B^{\pm}$ wrap about $\mu^{\pm}$ in $U^{\pm}$. This yields a genus two
    surface $\widetilde{U}_0$ with a geodesic lamination
    $\widetilde{\lambda}$ that contains $\mu^\pm$ as its stump and has two
    extra leaves each wrapping about $\mu^+$ in one direction and wrapping
    about $\mu^-$ in the other direction. 

    In each $A^{\pm}$ or $B^{\pm}$, there is a geodesic arc connecting the
    midpoint of $a^{\pm}$ to the midpoint of $b^{\pm}$. Because the gluing
    maps were reflections, the angles of these arcs with $a^{\pm}$ and
    $b^{\pm}$ match and the four arcs glue together to form a separating
    geodesic $\gamma$ in $\widetilde{U}_0$. Let $\ell_0$ be the length of
    $\gamma$. By construction, $\ell_0$ is independent of the irrational
    lamination $\mu$ and $U_0$. 

    Let $\widetilde{U}_t = \str(\widetilde{U},\widetilde{\lambda},t)$. Then
    $\widetilde{U}_t$ is also obtained from doubling of $U_t$ as above.
    Along the stretch paths, the length of geodesic arcs connecting the
    midpoint of $a^{\pm}$ to the midpoint of $b^{\pm}$ does not change.
    Hence, they glue the same way in $\widetilde{U}_t$ to form a separating
    geodesic of the same length. That is, the length of $\gamma$ along
    $\widetilde{U}_t$ remains the constant $\ell_0$. 
   
    The proposition now holds where $W_t$ is the subsurface of
    $\widetilde{U}_t$ with boundary $\gamma$. The stretch map from $U_s$ to
    $U_t$ doubles to an $e^{t-s}$--Lipschitz homeomorphism between
    $\widetilde{U}_s$ and $\widetilde{U}_t$ that fixes $\gamma$ pointwise.
    Hence, the length of curves grow by at most a factor of $e^{t-s}$ both
    from $\widetilde{U}_s$ to $\widetilde{U}_t$ and from $W_s$ to $W_t$. 

    To see the last assertion in the proposition, we note that there is
    contraction map from $W_t$ to $U_t$. Consider the restriction of $A^-$
    and $B^-$ in $W_t$ and foliate it with horocycles perpendicular to the
    boundary. Then collapse these regions sending each horocycle to a
    point. This is a distance decreasing map: the derivative in the
    direction tangent to the horocycles is zero and in the direction
    perpendicular to the horocycles is 1. Since stretch paths preserve
    these horocycles, this collapsing map commutes with the stretch maps
    and the image of $W_t$ under the collapsing map is exactly $U_t$.
    Hence, the length of a curve $\alpha$ in $W_t$ is longer than the
    length of its image in $U_t$ under the collapsing map, which is longer
    than the length of the geodesic representative of $\alpha$. \qedhere

  \end{proof}

   This proposition is the building block for all the examples we construct
   in this section. Essentially, we can glue $W_t$ to a family of surfaces
   with desired behavior to obtain various examples. 
       
   \begin{proof}[Proof of \thmref{Bad-Geodesic}]
     
     Let $\mu$ be a simple closed curve and let $U_0$ be a point in
     $\calT(S_{1,1})$ where the length of $\mu$ is $\ep$, for some small
     $\ep>0$. Let $W_t$ be the family obtained by \propref{Fixed-Boundary}.
     Also, choose $V \in \calT(S_{1,1}, \ell_0)$ to be a point in the thick
     part. Let $\beta$ be a curve of bounded length in $V$ and define $V^n=
     D_\beta^n V$, where $D_\beta$ is the Dehn twist around the curve
     $\beta$. (The values of $\ep$ and $n$ are to be determined below.)
   
     Let $s$ be the time when $\mu$ has length $1$ in $W_s$. Define
     \[
       X= W_0 \cup V, \qquad Y = W_s \cup V \qquad\text{and}\qquad Z=
       W_{s/2}\cup V^n. 
     \]
     By $\cup$ we mean glue the two surfaces along the boundary and
     consider them as an element of $\calT(S_{2,0})$. We mark the surfaces
     so that they have bounded relative twisting along the gluing curve
     $\gamma$. (Note that relative twisting is only well-defined up to an
     additive error). We claim that, if $\log 1/\ep \gg n$, then \[
       \dL(X,Z) = \dL(Z,Y) = s/2. 
     \]
     First consider $X$ and $Z$. Indeed, since the length of $\mu$ grows
     exponentially in $W_t$, $s/2$ is a lower bound for the distance
     $\dL(X,Z)$. We need to show that the length of any other curve grows
     by a smaller factor. This is true for any curve contained in $W_t$ by
     \propref{Fixed-Boundary}. For any other curve $\alpha$, let $\alpha_W$
     be the restriction of $\alpha$ to $W_0$ and let $\bar \alpha_W$ be the
     restriction to the complement of $W_0$. Any representative of $\alpha$
     in $Z$ is no shorter than the geodesic representative. Therefore,
     there is a uniform constant $C$ such that 
     \[
       \ell_Z(\alpha) \leq e^{s/2} \ell_X(\alpha_W) + C\big(\ell_X(\bar
       \alpha_W) + n \ell_X(\beta) \I(\alpha, \beta)\big). 
     \]
     But $\ell_X(\bar \alpha_W) \gmul \I(\alpha, \beta)$. Hence, if
     $e^{s/2} \gg C+n$, we obtain \[ e^{s/2} \ell_X(\bar \alpha_W) \ge (C+n)
     \big(\ell_X(\bar \alpha_W)+ \ell_X(\beta) \I(\alpha,\beta \big).\]
     Therefore, for sufficiently large $s$, we have, 
     \[
       \ell_Z(\alpha) \leq e^{s/2} \big( \ell_X(\alpha_W) + \ell_X(\bar
       \alpha_W) \big) = e^{s/2} \ell_X(\alpha). 
      \]
     That is, $\mu$ is the maximally stretched lamination from $X$ to $Z$. 
     The argument for the distance from $Z$ to $Y$ is similar. 
    
     Now define
     \[
       \GL_1(t) = W_t \cup V,
     \]
     and let $\GL_2(t)$ be the geodesic obtained by a concatenation of
     geodesic connecting $X$ to $Z$ and $Z$ to $Y$. Let $\alpha$ be a curve
     disjoint from $W_t$ that has a bounded length in $X$. Then for all $t$
     \[
       \ell_{\GL_1(t)}(\alpha) = \ell_X(\alpha) \emul 1
     \]
     But the length of $\alpha$ in $Z=\GL_2(s/2)$ is of order $n$. 
     That is, 
     \[
       \dL(\GL_1(t), Z)\geq \log
       \frac{\ell_Z(\alpha)}{\ell_{\GL_1(t)}(\alpha)} \eadd \log n, 
     \]
     which can be chosen to be much larger than $D$. Note that a lower
     bound for the distance in the other direction can also be found by
     replacing $\alpha$ with $D_\beta^{-n}(\alpha)$. 
     
     To obtain the second part of \thmref{Bad-Geodesic}, we note that the
     distance from $Y$ to $X$ is only of order $\log \log(1/\ep)$, if $\ep$
     is small enough. We now choose $n$ and $\ep$ such that,
     \[
       \log \frac 1\ep \gg \log n, \qquad \log n \eadd D \qquad D \gg \log
       \log \frac 1\ep. 
     \]
     This way, if $s \ge 2D$, then $Z=\GL_2(s/2)$ has distance at least $D$
     to any point on any geodesic connecting $Y$ to $X$. This finishes the
     proof of part 2. \qedhere

   \end{proof} 
   
   Next, we construct an example showing that the set of short curves in a
   Thurston geodesic connecting two points is not the same as the set of
   short curves along the \Teich geodesic. This is in contrast with the
   following theorem. 
   
   \begin{theorem}[\cite{LRT12}] \label{Thm:BoundedComb}
     
     For every $K>0$ and $\ep>0$ there exists $\ep'>0$ such that whenever
     $X, Y \in \T$ are $\ep$--thick and have $K$--bounded combinatorics,
     then any Thurston geodesic \GL from $X$ to $Y$ remains in the
     $\ep'$--thick part.

   \end{theorem}

   \begin{proof}[Proof of \thmref{Not-Short}]

     Let $\phi$ be a pseudo-Anosov map in the mapping class group of
     $S_{1,1}$ and let $U_0 \in \calT(S_{1.1})$ be on the \Teich axis of
     $\phi$. For any $n \in \ZZ$, the maximally stretched lamination $\mu$
     from $U_0$ to $\phi^n(U_0)$ is irrational, hence the stretch path from
     $U_0$ to $\phi^n(U_0)$ is the unique Thurston geodesic connecting
     them. Let $U_t = \st(U_0,\mu,t)$ and let $U_s = \phi^n(U_0)$. The
     point $\phi^n(U_0)$ is also on the \Teich axis of $\phi$ and the
     \Teich geodesic segment connecting $U_0$ to $\phi^n(U_0)$ stays in a
     uniform thick part (independent of $n$) of \Teich space. From
     \cite{LRT12} we know that the Thurston geodesic connecting these two
     points also stays in this part and fellow travels the \Teich geodesic.
     
     Let $W_t$ be the family of of surfaces in $\calT(S_{1.1},\ell_0)$
     obtained from \propref{Fixed-Boundary} and let $V$ be any point in
     the thick part of $\calT(S_{1.1},\ell_0)$. Now define 
     \[
       \GL(t) = W_t \cup V, \qquad t \in [0,s].
     \]
     Then $\GL(t)$ is a Thurston geodesic in $\calT(S_{2,0})$. This is
     because the length of $\mu$ is growing exponentially and the length of
     every other curve is growing by a smaller factor. (The argument is an
     easier version of the arguments in the previous proof and is dropped.)
    
     Since $U_t$ is in the thick part, so is $W_t$ and hence $\GL(t)$. Let
     $X=\GL(0)$ and $Y=\GL(s)$. But, by taking $n$ large enough, we can
     ensure that $d_{W}(X,Y)$ is as large as desired, where $W$ is the
     subsurface of $S_{2,0}$ associated to $W_t$. It then follows from
     \cite{Raf05} that the boundary of $W$ is short along the \Teich
     geodesic connecting $X$ to $Y$, in fact, its minimum length is
     inversely proportional to $d_W(X,Y)$. That is, $\partial W$ has
     bounded length along the Thurston geodesic $\GL(t)$ but is arbitrary
     short along the \Teich geodesic. This finishes the proof of the
     theorem. \qedhere 

  \end{proof} 
   
 
  \bibliographystyle{alpha}
  \bibliography{main}

\begin{thebibliography}{BMM11}

\bibitem[Aou13]{Aou13}
T.~Aougab.
\newblock Uniform hyperbolicity of the graphs of curves.
\newblock {\em Geom. Topol.}, 17(5):2855--2875, 2013.

\bibitem[BF11]{BF11}
M.~Bestvina and M.~Feighn.
\newblock Hyperbolicity of the complex of free factors.
\newblock Preprint, 2011.
\newblock Available at {\tt arXiv:1107.3308 [math.GR]}.

\bibitem[BMM11]{BMM11}
J.~Brock, H.~ Masur, and Y.~Minsky.
\newblock Asymptotics of {W}eil-{P}etersson geodesics {II}: bounded geometry
  and unbounded entropy.
\newblock {\em Geom. Funct. Anal.}, 21(4):820--850, 2011.

\bibitem[Bon97]{Bon97}
F.~Bonahon.
\newblock Transverse {H}\"older distributions for geodesic laminations.
\newblock {\em Topology}, 36(1):103--122, 1997.

\bibitem[Bow06]{Bow06}
B.~H.~Bowditch.
\newblock Intersection numbers and the hyperbolicity of the curve complex.
\newblock {\em J. Reine Angew. Math.}, 598:105--129, 2006.

\bibitem[Bow13]{Bow13}
B.~H.~Bowditch.
\newblock Uniform hyperbolicity of the curve graphs.
\newblock To appear in Pacific J. Math, 2013.

\bibitem[CDR12]{CDR12}
Y.-E.~Choi, D.~Dumas, and K.~Rafi.
\newblock Grafting rays fellow travel {T}eichm\"uller geodesics.
\newblock {\em Int. Math. Res. Not. IMRN}, (11):2445--2492, 2012.

\bibitem[CRS08]{CRS08}
Y.-E. Choi, K.~Rafi, and C.~Series.
\newblock Lines of minima and {T}eichm\"uller geodesics.
\newblock {\em Geom. Funct. Anal.}, 18(3):698--754, 2008.

\bibitem[CRS13]{CRS13}
M.~Clay, K.~Rafi, and S.~Schleimer.
\newblock Uniform hyperbolicity of the curve graph via surgery sequences.
\newblock To appear in Algebraic and Geometric Topology, 2013.
\newblock Available at {\tt arXiv:1302.5519 [math.GT]}.

\bibitem[GK13]{GK13}
F.~Gu\'eritaud and F.~Kassel.
\newblock Maximally stretched laminations on geometrically finite hyperbolic
  manifolds.
\newblock To appear in Geom. Topol., 2013.
\newblock Available at {\tt arXiv:1307.0250 [math.GT]}.

\bibitem[Har81]{Har81}
W.~J.~Harvey.
\newblock Boundary structure of the modular group.
\newblock In {\em Riemann surfaces and related topics: {P}roceedings of the
  1978 {S}tony {B}rook {C}onference ({S}tate {U}niv. {N}ew {Y}ork, {S}tony
  {B}rook, {N}.{Y}., 1978)}, volume~97 of {\em Ann. of Math. Stud.}, pages
  245--251. Princeton Univ. Press, Princeton, N.J., 1981.

\bibitem[HPW13]{HPW13}
S.~Hensel, P.~Przytycki, and R.~C.~H.~Webb.
\newblock Slim unicorns and uniform hyperbolicity for arc graphs and curve
  graphs.
\newblock To appear in J. Eur. Math. Soc., 2013.
\newblock Available at {\tt arXiv:math/1301.5577 [math.GT]}.

\bibitem[Hub06]{Hub06}
J.~H.~Hubbard.
\newblock {\em Teichm\"uller theory and applications to geometry, topology, and
  dynamics. {V}ol. 1}.
\newblock Matrix Editions, Ithaca, NY, 2006.
\newblock Teichm{\"u}ller theory, With contributions by Adrien Douady, William
  Dunbar, Roland Roeder, Sylvain Bonnot, David Brown, Allen Hatcher, Chris
  Hruska and Sudeb Mitra, With forewords by William Thurston and Clifford
  Earle.

\bibitem[LRT12]{LRT12}
A.~Lenzhen, K.~Rafi, and J.~Tao.
\newblock Bounded combinatorics and the {L}ipschitz metric on {T}eichm\"uller
  space.
\newblock {\em Geom. Dedicata}, 159:353--371, 2012.

\bibitem[Min96]{Min96a}
Y.~N.~Minsky.
\newblock Extremal length estimates and product regions in {T}eichm\"uller
  space.
\newblock {\em Duke Math. J.}, 83(2):249--286, 1996.

\bibitem[Min10]{Min10}
Y.~Minsky.
\newblock The classification of {K}leinian surface groups. {I}. {M}odels and
  bounds.
\newblock {\em Ann. of Math. (2)}, 171(1):1--107, 2010.

\bibitem[MM99]{MM99}
H.~A.~Masur and Y.~N.~ Minsky.
\newblock Geometry of the complex of curves. {I}. {H}yperbolicity.
\newblock {\em Invent. Math.}, 138(1):103--149, 1999.

\bibitem[MM00]{MM00}
H.~A.~Masur and Y.~N.~Minsky.
\newblock Geometry of the complex of curves. {II}. {H}ierarchical structure.
\newblock {\em Geom. Funct. Anal.}, 10(4):902--974, 2000.

\bibitem[Pap07]{Pap07}
A.~Papadopoulos.
\newblock Introduction to {T}eichm\"uller theory, old and new.
\newblock In {\em Handbook of {T}eichm\"uller theory. {V}ol. {I}}, volume~11 of
  {\em IRMA Lect. Math. Theor. Phys.}, pages 1--30. Eur. Math. Soc., Z\"urich,
  2007.

\bibitem[Raf05]{Raf05}
K.~Rafi.
\newblock A characterization of short curves of a {T}eichm\"uller geodesic.
\newblock {\em Geom. Topol.}, 9:179--202, 2005.

\bibitem[Sch06]{Sch06}
S.~Schleimer.
\newblock Notes on the complex of curves.
\newblock Preprint, 2006.
\newblock Available online at {\tt
  http://homepages.warwick.ac.uk/$\sim$masgar/Maths/notes.pdf}.

\bibitem[Thu86]{Thu86a}
W.~P.~Thurston.
\newblock A norm for the homology of {$3$}-manifolds.
\newblock {\em Mem. Amer. Math. Soc.}, 59(339):i--vi and 99--130, 1986.

\end{thebibliography}

  \end{document}